\documentclass[12pt,a4paper]{amsart}

\usepackage{graphics,epic}
\usepackage{amsmath,amssymb, amsthm}
\usepackage[all,2cell]{xy}
\usepackage{mathrsfs}
\usepackage{color}

\newtheorem{theorem}{Theorem}[section]
\newtheorem*{theorem*}{Theorem}

\newtheorem{lemma}[theorem]{Lemma}
\newtheorem{proposition}[theorem]{Proposition}

\newtheorem*{conjecture*}{Conjecture}

\newtheorem{remark}[theorem]{Remark}

\newcommand{\opname}[1]{\operatorname{\mathsf{#1}}}

\renewcommand{\mod}{\opname{mod}\nolimits}

\newcommand{\add}{\opname{add}\nolimits}

\newcommand{\dimv}{\underline{\dim}\,}

\newcommand{\ind}{\opname{ind}}

\newcommand{\rep}{\opname{rep}\nolimits}

\newcommand{\Z}{\mathbb{Z}}
\newcommand{\N}{\mathbb{N}}
\newcommand{\Q}{\mathbb{Q}}

\newcommand{\id}{\mathbf{1}}

\newcommand{\Hom}{\opname{Hom}}
\newcommand{\go}{\opname{G_0}}

\newcommand{\Ext}{\opname{Ext}}

\newcommand{\Aut}{\opname{Aut}}

\begin{document}

\title[Feigin's map revisited]{Feigin's map revisited}

\author{Changjian Fu}
\address{Department of Mathematics\\
SiChuan University\\
610064 Chengdu\\
P.R.China
}
\email{
\begin{minipage}[t]{5cm}
changjianfu@scu.edu.cn
\end{minipage}}

\subjclass[2010]{16G20, 17B37}
\keywords{Feigin's map, quantum enveloping algebra,  Ringel-Hall algebra,  quantum shuffle algebra, monomial basis}

\begin{abstract}
The aim of this note is to understand the injectivity of Feigin's map $\mathbf{F_w}$ by representation theory of quivers, where $\mathbf{w}$ is the word of a reduced expression of the longest element of a finite Weyl group. This is achieved by the Ringel-Hall algebra approach and a careful studying of a well-known total order on the category of finite-dimensional representations of a valued quiver of finite type.
 As a byproduct, we also generalize Reineke's construction of monomial bases to non-simply-laced cases.
\end{abstract}
\maketitle

\section{Introduction}
Let $U_+$ be the positive part of a quantum enveloping algebra associated to a Cartan matrix $C\in M_{I\times I}(\Z)$ and ${\bf w}=(i_1i_2\cdots i_m)\in I^m$ a word in alphabet $I$. In 1990s', B. Feigin proposed a homomorphism $\mathbf{F_w}$ from $U_+$ to the quantum polynomial algebra $\mathbb{P}_{\bf w}$ as a tool for studying the skew-field of fractions of $U_+$. He also conjectured that $\mathbf{F_w}$ is an embedding provided that $C$ is of finite type and ${\bf w}$ is a word associated to a reduced expression of the longest element in the Weyl group of $C$. This conjecture has been confirmed by K. Iohara and F. Malikov in a special case~\cite{IM} and by A. Joseph in general~\cite{Joseph}. Moreover, $\mathbf{F_w}$ induces an isomorphism of skew-fields of fractions on both sides in this case (see also ~\cite{Berenstein}).
In particular, this embedding gives us a realization of $U_+$ as subalgebra of $\mathbb{P}_{\bf w}$ which leads to a construction of (dual) monomial bases for the quantum enveloping algebras of finite simply-laced type in~\cite{Reineke}.

There are two other well-known realizations of $U_+$ which are independent of the choice of the type of $C$. The first one is an embedding of $U_+$ into the  quantum shuffle algebra~\cite{Green97, Rosso}, which is dual to the realization of $U_+$ as a quotient of a free algebra. The second one is an embedding of $U_+$ into (dual) Ringel-Hall algebra, which was  discovered by Ringel~\cite{Ringel} and Green~\cite{Green}. Both  realizations provided us a better understanding of certain  bases of $U_+$~({\it cf.}~\cite{Leclerc, Ringel, Rosso} for instance). Recently, the Feigin's map $\mathbf{F_w}$ has been extended to  quantum shuffle algebras in~\cite{Rupel} and (dual) Ringel-Hall algebras in~\cite{BR}. Moreover, these Feigin-type maps fit into a commuting tetrahedron of maps beginning with the quantum enveloping algebra and terminating in a quantum polynomial algebra~\cite{Rupel}.

Let $C$ be a Cartan matrix of finite type and $(Q, {\bf d})$ a valued quiver associated to $C$. Let $\mathcal{H}^*(Q)$ be the dual Ringel-Hall algebra of $(Q, {\bf d})$. In this case, the embedding of $U_+$ into $\mathcal{H}^*(Q)$ is an isomorphism.
In particular, $U_+$ admits a PBW basis parametrized by the isoclasses of representations of $(Q, {\bf d})$. Let $\omega_0$ be the longest element of the Weyl group associated to $C$.
 It is well-known that the reduced expressions of $\omega_0$ also have close connection with the representation theory of $(Q, {\bf d})$ (see Section~\ref{s:Weylgroup} or~\cite{Lusztig90} for instance). Our main purpose  is to pursue a  representation-theoretic understanding of the injectivity of Feigin's map $\mathbf{F_w}$, where ${\bf w}$ is a word associated to a reduced expression of $\omega_0$. This is achieved  by a careful studying  of a well-known total order on the category of finite-dimensional representations of a  valued quiver of finite type. In fact, we obtain the injectivity of $\mathbf{F_w}$ for a large class of words ${\bf w}$ which may not be the words of reduced expressions of the longest element $\omega_0$~({\it cf.} Theorem~\ref{t:injective-directed-partition}). Moreover, we find that the total order consider in present paper can be used to replace the degeneration order $\leq_{\deg}$ in~\cite{Reineke}. Thus we  generalize the construction of monomial bases in~\cite{Reineke} to all the Cartan matrices of finite type.

The paper is structured as follows. In Section~\ref{s:preliminary},
we recall definitions and  basic properties concerning graded duals of graded (co)algebras, quantum enveloping algebras, quantum shuffle algebras and quantum polynomial algebras.
Section~\ref{s:Feigin-typemaps} is devoted to give a unified treatment of the Feigin-type maps by using the graded dual approach, which simplifies a lot of calculations in~\cite{Rupel}. The results will be employed to give a representation-theoretic understanding of the injectivity of the Feigin's map in Section~\ref{s:Feiginmap}. After introducing the total order associated to an enumeration in Section~\ref{s:order}, we prove a fundamental result (Proposition~\ref{p:order}) on the behavior of a representation with respect to the total order in Section~\ref{s:directedpartition}.  We then deduce the injectivity  of Feigin's map (Theorem~\ref{t:injective-directed-partition} and ~\ref{t:injective}) from Proposition~\ref{p:order}. We end up with two applications of the total order in Section~\ref{s:consequence}. The first one is to generalize the construction of ~\cite{Reineke} to all the Cartan matrices of finite type. In particular, we obtain  monomial bases~(Theorem~\ref{t:monomialbase}) for quantum enveloping algebras of finite type following ~\cite{Reineke}. The second one~(Proposition~\ref{p:characterization}) is a  characterization of modules for representation-finite hereditary algebras over finite field.

\noindent{\bf Acknowledgements.} 
The author would like to thank the referee for his/her valuable comments and suggestions in making this article more readable.
This work was partially support by NSF of China (No. 11471224).

\section{Preliminary}~\label{s:preliminary}
\subsection{Quantum binomial coefficients}~\label{ss:quantum-binomial}
Let $v$ be an indeterminate. The quantum numbers, factorials and binomials in variable $v$ are defined as follows
\[[n]:=v^{-n+1}+v^{-n+3}+\cdots+v^{n-1}, [n]!:=[n][n-1]\cdots [2][1], \left[\begin{array}{c}n\\ k \end{array}\right]=\frac{[n]!}{[k]![n-k]!}.
\]
The following result is known as Pascal identities.
\begin{lemma}~\label{l:pascal-identities}
\[\left[\begin{array}{c}n\\ k\end{array}\right]=v^{n-k}\left[\begin{array}{c}n-1\\ k-1\end{array}\right]+v^{-k}\left[\begin{array}{c}n-1\\ k\end{array}\right].
\]
\end{lemma}

\subsection{Recollection on bialgebras and their graded duals}
We refer to~\cite{Kassel}~for basic definitions and properties of coalgebras and bialgebras.
Let $K$ be a field and $I$ a finite set. For a finite-dimensional $K$-vector space $V$, denote by $V^*=\Hom_K(V,K)$ the dual space of $V$ and $\langle-,-\rangle_V: V^*\times V\to K$ the evaluation pairing.   The evaluation pairing $\langle-,-\rangle_V$ induces a pairing
$\langle-,-\rangle_{V\otimes V}$ between $V^*\otimes V^*$ and $V\otimes V$ by setting
\[\langle x\otimes y, a\otimes b\rangle_{V\otimes V}=\langle x,a\rangle_V\langle y,b\rangle_V,
\]
where $x, y\in V^*$ and $a,b\in V$.

Let $(A, \mu, \eta)$ be a $K$-algebra with multiplication $\mu: A\otimes A\to A$ and unit $\eta:K\to A$. The algebra $(A, \mu, \eta)$  is an $\N I$-graded algebra if there is a gradation on $A=\bigoplus\limits_{\alpha\in \N I}A_\alpha$ such that $\eta(K)\subseteq A_0$ and $\mu(A_\alpha\otimes A_\beta)\subseteq A_{\alpha+\beta}$ for each $\alpha,\beta\in \N I$.

Let $(A=\bigoplus\limits_{\alpha\in \N I}A_\alpha, \mu, \eta)$ be an $\N I$-graded $K$-algebra with $\dim_KA_\alpha<\infty$ for each $\alpha\in \N I$. Let $A^*=\bigoplus\limits_{\alpha\in \N I}A_\alpha^*$ be the graded dual space of $A$ and $\langle-, -\rangle_A: A^*\times A\to K$ the associated evaluation pairing.
Note that we have $\langle A^*_{\alpha}, A_\beta\rangle_A=\{0\}$ for $\alpha\neq \beta\in \mathbb{N}I$.
 Define the $K$-linear map $\mu^*:A^*\to A^*\otimes A^*$ by \[\langle \mu^*(c), a\otimes b\rangle_{A\otimes A}=\langle c,\mu(a\otimes b)\rangle_A \] for homogeneous elements $a,b\in A$ and $c\in A^*$.  We call $\mu^*$ the {\it adjoint of $\mu$} with respect to the evaluation pairing $\langle-,-\rangle_A$. Similarly, one defines $\eta^*$ to be the adjoint of $\eta$.  It is well-known that $(A^*, \mu^*, \eta^*)$ is an $\N I$-graded coalgebra over $K$. Namely, $(A^*, \mu^*, \eta^*)$ is a coalgebra such that
\[\mu^*(A_\alpha^*)\subseteq \bigoplus_{\alpha_1+\alpha_2=\alpha}A_{\alpha_1}^*\otimes A_{\alpha_2}^*\] and $\eta^*(A_\alpha^*)=0$ for $\alpha\neq 0$.

Dually, let $(A=\bigoplus\limits_{\alpha\in \N I}A_\alpha, \Delta, \epsilon)$ be an $\N I$-graded coalgebra with comultiplication $\Delta: A\to A\otimes A$ and counit $\epsilon: A\to K$. Assume that $\dim_KA_\alpha<\infty$ for each $\alpha\in \mathbb{N}I$. One can also consider its graded dual space $A^*$ and the adjoints $\Delta^*, \epsilon^*$ of $\Delta$ and $\epsilon$ respectively. Then $(A^*, \Delta^*, \epsilon^*)$ is an $\N I$-graded $K$-algebra.
In particular, if $(A=\bigoplus\limits_{\alpha\in \N I}A_\alpha, \mu, \eta, \Delta, \epsilon)$ is an $\N I$-graded bialgebra  over $K$ with finite-dimensional components, then its graded dual $(A^*=\bigoplus\limits_{\alpha\in \N I}A_\alpha^*, \Delta^*, \epsilon^*, \mu^*,\eta^*)$ is an $\N I$-graded bialgebra over $K$.

The $\N I$-graded algebras, coalgebras and bialgebras considered in the following are assumed to have finite-dimensional components.

Let $A=\bigoplus\limits_{\alpha\in \N I}A_\alpha$ and $B=\bigoplus\limits_{\alpha\in \N I}B_\alpha$ be $\N I$-graded $K$-algebras.
A $K$-linear map $f: A\to B$ is a homomorphism  of $\N I$-graded $K$-algebras if $f$ is a homomorphism of $K$-algebras and $f(A_\alpha)\subseteq B_\alpha$ for each $\alpha\in \N I$.  For a given homomorphism $f: A\to B$ of $\N I$-graded algebras, we consider the linear map $f^*: B^*\to A^*$ from the graded dual $B^*$ to $A^*$ defined by
\[\langle f^*(b^*), a\rangle_A=\langle b^*, f(a)\rangle_B, \forall~ b^*\in B_\alpha^*, a\in A_\alpha.
\]
We call $f^*$ the {\it adjoint} of $f$ with respect to the evaluation pairings $\langle-, -\rangle_A$ and $\langle-,-\rangle_B$.  It is easy to check that $f^*: B^*\to A^*$ is a homomorphism of $\N I$-graded coalgebras. That is, $f^*$ is a homomorphism of colagebras and $f^*(B_\alpha^*)\subseteq A_\alpha^*$ for each $\alpha\in \N I$.

Dually, if $f: A=\bigoplus\limits_{\alpha\in \N I}A_\alpha\to B=\bigoplus\limits_{\alpha\in \N I}B_\alpha$ is a homomorphism of $\N I$-graded coalgebras, then $f^*:B^*\to A^*$ is a homomorphism of $\N I$-graded algebras. In particular, if $f: A\to B$ is a homomorphism of $\N I$-graded bialgebras, then $f^*:B^*\to A^*$ is a homomorphism of $\N I$-graded bialgebras.  The following fact is well-known ({\it cf.} Proposition 2.5 in~\cite{Green97}).
\begin{proposition}~\label{p:non-degenerate-bilinear-form}
Let $(A=\bigoplus\limits_{\alpha\in \mathbb{N}I}A_\alpha,\mu,\eta,\Delta,\epsilon)$ be an $\N I$-graded bialgebra over $K$ with finite-dimensional components. Assume that there is a non-degenerate symmetric bilinear form $(-,-)_A:A\times A\to K$ such that 
\[(a\otimes b, \Delta(c))_{A\otimes A}=(\mu(a\otimes b), c)_A,~ (\eta(1), a)_A=\epsilon(a)\] 
for any $a,b,c\in A$ and 
$(A_\alpha,A_\beta)=0$ whenever $\alpha\neq \beta\in \mathbb{N}I$.
Then $(A,\mu,\eta,\Delta,\epsilon)$ is isomorphic to $(A^*, \Delta^*,\epsilon^*,\mu^*,\eta^*)$ as $\N I$-graded bialgebras.
\end{proposition}
\begin{proof}
Let $\psi$ be the linear map defined as follows
\begin{eqnarray*}
\psi: &A&\to A^*.\\
&a&\mapsto (a,-)_A
\end{eqnarray*}
It is clear that $\psi$ is an isomorphism of $K$-vector spaces. Let $\langle-,-\rangle_A:A^*\times A\to K$ be the evaluation pairing. For any $a,b \in A$, we clearly have $\langle \psi(a), b\rangle_A=(a,b)_A$. Now it is routine to check that $\psi$ is a homomorphism of bialgebras. For instance, for any $a,b,c\in A$,  we have
\begin{eqnarray*}\langle \Delta^*\circ \psi\otimes \psi(a,b), c\rangle_{A\otimes A}&=&\langle \psi\otimes \psi(a\otimes b), \Delta(c)\rangle_{A\otimes A}\\
&=&\langle \psi(a)\otimes \psi(b), \Delta(c)\rangle_{A\otimes A}\\
&=&(a\otimes b, \Delta(c))_{A\otimes A}.
\end{eqnarray*}
Consequently, 
\[\langle \psi\circ \mu(a\otimes b), c\rangle_A=(\mu(a\otimes b), c)_A=(a\otimes b, \Delta(c))_{A\otimes A}.
\]
On the other hand, we clearly have 
\[\langle \psi\circ \eta(1), a\rangle_A=(\eta(1), a)_A=\epsilon(a)=\langle\epsilon^*(1),a\rangle_A.
\]
In particular,  $\psi$ satisfies the following commutative diagrams
\[\xymatrix{A\otimes A\ar[d]^{\mu}\ar[rr]^{\psi\otimes \psi}&&A^*\otimes A^*\ar[d]^{\Delta^*}\\
A\ar[rr]^{\psi}&&A^*}~~\xymatrix{K\ar[d]^{\eta}\ar[dr]^{\epsilon^*}\\
A\ar[r]^{\psi}&A^*.}
\]
The other commutative diagrams can be verified similarly.
\end{proof}

\subsection{The free algebra $\mathcal{F}$ and quantum shuffle algebra} We follow~\cite{Lusztig93}.
Let $C=(c_{ij})\in M_{I\times I}(\Z)$ be a symmetrizable Cartan matrix with symmetrizer ${\bf f}=\operatorname{diag}\{f_i, ~i\in I\}$.
Let $L$ be the root lattice of the Kac-Moody algebra $\mathfrak{g}$ associated to the Cartan matrix $C$. Denote by  $\Pi=\{\alpha_j,~j\in I\}$ the set of simple roots of $\mathfrak{g}$ which forms a $\Z$-basis of $L$.  Let $(-,-):L\times L\to \Z$ be the symmetric bilinear form  defined by $(\alpha_i,\alpha_j)=f_ic_{ij}$ for $i,j\in I$. We identify $\N I$ with $L^+:=\{\sum_{i\in I}n_i\alpha_i~|~n_i\geq 0,~i\in I\}$.

Let $v$ be an indeterminate.  The {\it free algebra $\mathcal{F}$} is the free $\Q(v)$-algebra generated by $x_j,~j\in I$. If we set $\deg x_j=|x_j|=\alpha_j\in L$, then $\mathcal{F}=\bigoplus\limits_{\alpha\in \N I}\mathcal{F}_\alpha$ is an $\N I$-graded algebra such that $\dim_{\Q(v)}\mathcal{F}_\alpha<\infty$ for each $\alpha\in \N I$. In particular, $\mathcal{F}_0=\mathbb{Q}(v)$. Denote by $\mu$ the multiplication of $\mathcal{F}$ and $\eta:\Q(v)\to \mathcal{F}$ the unit of $\mathcal{F}$.

We endow $\mathcal{F}\otimes \mathcal{F}$ with the following twisted multiplication
\[(u_1\otimes z_1)(u_2\otimes z_2):=v^{(|z_1|, |u_2|)}u_1u_2\otimes z_1z_2
\]
for homogeneous elements $u_1,u_2,z_1,z_2$ of $\mathcal{F}$. 
Denote by $\Delta: \mathcal{F}\to \mathcal{F}\otimes \mathcal{F}$ the homomorphism of algebras given by \[\Delta(x_j)= x_j\otimes 1+1\otimes x_j~ \text{for} ~j\in I.\]  Let $\epsilon:\mathcal{F}\to \Q(\nu)$ be the $\mathbb{Q}(v)$-linear map defined by $\epsilon(\mathcal{F}_\alpha)=0$ for $\alpha\neq 0$ and $\epsilon(a)=a$ for $a\in \mathcal{F}_0$. Then $(\mathcal{F}, \mu,\eta, \Delta, \epsilon)$
 is an $\N I$-graded bialgebra.
 Set $v_i=v^{f_i},~i\in I$.
We may define the quantum numbers $[n]_i$, factorials $[n]_i^!$ and binomials $\left[\begin{array}{c}r\\ k\end{array}\right]_i$ in the variable $v_i$ as in Section~\ref{ss:quantum-binomial}.
 The comulitplication on the monomials are easy to determine.

 \begin{lemma}~\label{l:computation}
 \begin{itemize}
 \item[(1)]  Let $i\in I,~n\in \N$,
 \[\Delta(x_i^n)=\sum_{k=0}^n
v_i^{(n-k)k}\left[\begin{array}{c}n\\ k\end{array}\right]_ix_i^{n-k}\otimes x_i^k; \]
\item[(2)] Let $i_1,\cdots, i_m\in I$ and $\underline{a}=(a_1,\cdots, a_m)\in \N^m$,
\begin{eqnarray*}&&\Delta(x_{i_1}^{a_1}\cdots ~x_{i_m}^{a_m})\\
&=&\sum_{\underline{b}+\underline{c}=\underline{a}}(\prod_{k=1}^m~v_{i_k}^{b_kc_k})v^{\sum_{k<l}c_kb_l(\alpha_{i_k}, \alpha_{i_l})}\left[\begin{array}{c}a_1\\ c_1\end{array}\right]_{i_1}\cdots \left[\begin{array}{c}a_m\\ c_m\end{array}\right]_{i_m}x_{i_1}^{b_1}\cdots x_{i_m}^{b_m}\otimes x_{i_1}^{c_1}\cdots x_{i_m}^{c_m}.
\end{eqnarray*}
 \end{itemize}
 \end{lemma}
 \begin{proof}
 The statement $(1)$ can be proved by induction on $n$ and  the Pascal identities. The statement $(2)$ follows from $(1)$ and the fact that $\Delta$ is a homomorphism of algebras.
 \end{proof}

 The {\it quantum shuffle algebra}~\cite{Green97} associated to the Cartan matrix $C$ is the graded dual
 $(\mathcal{F}^*, \Delta^*,\epsilon^*,\mu^*,\eta^*)$ of $(\mathcal{F}, \mu,\eta, \Delta, \epsilon)$ which is  again an $\N I$-graded bialgebra.
 Let $W=\cup_{m\geq 0}I^m$ be the set of words in alphabet $I$. For each ${\bf u}=(j_1j_2\cdots j_m)\in W$, set $x_{\bf u}=x_{j_1}x_{j_2}\cdots x_{j_m}\in \mathcal{F}$. It is clear that $\{x_{\bf u}~|~{\bf u}\in W\}$ forms a $\Q(v)$-basis of $\mathcal{F}$. Let $\{y_{\bf u}~|~{\bf u}\in W\}$ be the basis of $\mathcal{F}^*$ dual to $\{x_{\bf u}~|~{\bf u}\in W\}$. Then it is easy to compute the comultiplicatioin $\mu^*$ of $\mathcal{F}^*$:
 \[\mu^*(y_{\bf u})=\sum_{({\bf u_1}, {\bf u_2})={\bf u} }y_{\bf u_1}\otimes y_{\bf u_2}.
 \]
In order to compute the multiplication $\Delta^*$ of $\mathcal{F}^*$, we need some more notation. Let $\Sigma_r$ be the symmetry group of $r$ letters. We denote an element $\sigma\in \Sigma_r$ by $\sigma=(\sigma_1,\cdots, \sigma_r)$, where $\sigma_1,\cdots, \sigma_r$ is a permutation of $1,2,\cdots, r$. Let $\sigma_k^{-1}$ be the preimage of $k$ under $\sigma$ and define
\[\Sigma_{r,s}:=\{\sigma\in \Sigma_{r+s}~|~\sigma_1^{-1}<\cdots<\sigma_r^{-1}~\text{and}~\sigma_{r+1}^{-1}<\cdots<\sigma_{r+s}^{-1}\}.
\]
Using the definition of $\Delta^*$,  for ${\bf u_1}=(j_1,\cdots, j_r),~{\bf u_2}=(j_{r+1},\cdots, j_{r+s})\in W$, we have
\[\Delta^*(y_{\bf u_1}\otimes y_{\bf u_2})=\sum_{\sigma\in \Sigma_{r,s}}v^{s(\sigma,{\bf u_1}, {\bf u_2})}y_{\sigma({\bf u_1},{\bf u_2})},
\]
where $s(\sigma,{\bf u_1},{\bf u_2})=\sum_{k\leq r,l>r, \sigma_k^{-1}>\sigma_l^{-1}}(\alpha_{j_k}, \alpha_{j_l})$ and $\sigma({\bf u_1},{\bf u_2})=(j_{\sigma_1}, j_{\sigma_2},\cdots, j_{\sigma_{r+s}})$.

\subsection{Quantum enveloping algebra}
Recall that $C=(c_{ij})\in M_{I\times I}(\Z)$ is a symmetrizable Cartan matrix.
The {\it quantum enveloping algebra} $U_+$ is the $\Q(v)$-algebra generated by $E_j, ~j\in I$ subject to the quantum Serre relations
\[ \sum_{r=0}^{1-c_{ij}}(-1)^rE_i^{[r]}E_jE_i^{[1-c_{ij}-r]}=0, ~ \text{for}~ i\neq j,
\]
where $E_i^{[r]}=\frac{1}{[r]_i^{!}}E_i^r$. Define $\deg E_j=|E_j|=\alpha_j\in L$, then $U_+=\bigoplus\limits_{\alpha\in \N I}(U_+)_\alpha$ is an $\N I$-graded algebra which is a quotient algebra of $\mathcal{F}$.

In fact, the algebra $U_+$ has a bialgebra structure whose comultiplication  $\Delta:U_+\to U_+\otimes U_+$ is given by $\Delta(E_i)=E_i\otimes 1+1\otimes E_i,~i\in I$ and the counit $\epsilon: U_+\to \Q(v)$ is given by $\epsilon(E_i)=0,~i\in I$. The canonical projection $\mathcal{F}\to U_+, x_i\mapsto E_i,~i\in I$ is a homomorphism of bialgebras. Moreover, there is a non-degenerate symmetric bilinear form $(-,-)_{U_+}:U_+\times U_+\to \Q(v)$ satisfying the condition of Proposition~\ref{p:non-degenerate-bilinear-form}.
Therefore the quantum enveloping algebra $U_+$ is isomorphic to its graded dual $U_+^*$.
In particular, let $e_i,~i\in I$ be the basis of $(U_+^*)_{\alpha_i}$ dual to $E_i\in (U_+)_{\alpha_i}$, then $\pi:\mathcal{F}\to U_+^*, x_i\mapsto e_i$ is a homomorphism of bialgebras. For more properties of $U_+$, we refer to ~\cite{Lusztig93}.

\subsection{Quantum polynomial algebras}
Let ${\bf w}=(i_1i_2\cdots i_m)\in W$ be a fixed word. The {\it quantum polynomial algebra} $\mathbb{P}_{\bf w}$ associated to the Cartan matrix $C$ and the word ${\bf w}$ is the $\Q(v)$-algebra generated by $t_1,\cdots, t_m$ subject to the relation \[t_lt_k=v^{(\alpha_{i_k}, \alpha_{i_l})}t_kt_l~\text{ for $k<l$}.\]

For $\underline{a}=(a_1,a_2,\cdots, a_m)\in \N^{m}$, set $t^{\underline{a}}:=t_1^{a_1}t_2^{a_2}\cdots t_m^{a_m}$. It is clear that $\{t^{\underline{a}}~|~\underline{a}\in \N^m\}$ is a $\Q(v)$-basis of $\mathbb{P}_{\bf w}$ and we have
\[t^{\underline{a}}t^{\underline{b}}=v^{\sum_{k<l}b_ka_l(\alpha_{i_k}, \alpha_{i_l})}t^{\underline{a}+\underline{b}}.
\]

Set $\deg t_k=|t_k|=\alpha_{i_k}, ~1\leq k\leq m$. Then the  algebra $\mathbb{P}_{\bf w}$ is $\N I$-graded and each component is finite-dimensional.
Let $\mathbb{P}_{\bf w}^*$ be the graded dual of $\mathbb{P}_{\bf w}$ which is a coalgebra. Denote by $\Delta=\mu^*$ and $\epsilon=\eta^*$ respectively  the comultiplication and counit of $\mathbb{P}_{\bf w}^*$. Let $\{t_{\underline{a}}~|~\underline{a}\in \N^m\}$ be the basis of $\mathbb{P}_{\bf w}^*$ dual to $\{t^{\underline{a}}~|~\underline{a}\in \N^m\}$. A direct calculation shows that
\begin{lemma}
For $\underline{a}\in \N^m$,
\[\Delta(t_{\underline{a}})=\sum_{\underline{b}+\underline{c}=\underline{a}}v^{\sum_{k<l}c_kb_l(\alpha_{i_k}, \alpha_{i_l})}t_{\underline{b}}\otimes t_{\underline{c}}.
\]
\end{lemma}

\section{Feigin-type maps}~\label{s:Feigin-typemaps}
\subsection{Reminder on valued quiver}
A {\it valued graph}  is a pair $(I, {\bf d})$ consisting of a finite set $I$ of vertices together with non-negative integers $d_{ij}$ for all $i, j\in I$ such that $d_{ii}=0$ and there exists positive integers $\{f_i, ~i\in I\}$ satisfying $f_id_{ij}=f_jd_{ji}$ for all $i, j\in I$.
A pair $(i, j)\in I\times I$ is an {\it edge} of the valued graph $(I, {\bf d})$ provided $d_{ij}\neq 0$.
An {\it orientation} $\Omega$ of a valued graph $(I, {\bf d})$ is to assign each edge $(i, j)$ an arrow $i\to j$ or $j\to i$. In this case, we obtain a quiver $Q$ with vertex set $I$.
A {\it valued quiver} is a valued graph $(I, {\bf d})$ endowed with an orientation $\Omega$. In the following, we always denote a valued quiver by $(Q, {\bf d})$ and assume that $Q$ is acyclic.

Let $K$ be a field and $(Q, {\bf d})$ a valued quiver. A {\it reduced $K$-species} of $(Q, {\bf d})$ is a pair $\mathcal{M}=\{\mathbb{F}_{i}, {}_{i}M_{j}\}_{i, j\in I}$ consisting of division rings $\mathbb{F}_{i}$ over $K$ with $\dim_{K}\mathbb{F}_{i}=f_{i}$ for $i\in I$ and an $\mathbb{F}_{i}$-$\mathbb{F}_{j}$-bimodule ${}_{i}M_{j}$ with $\dim_{K}~_{i}M_{j}=f_id_{ij}$ for each arrow $i\to j$.
A finite-dimensional $K$-representation $V=(V_{i}, \varphi_{\alpha})$ of $(Q, {\bf d})$ consists of a finite-dimensional $\mathbb{F}_{i}$-vector space $V_{i}$ for $i\in I$ and an $\mathbb{F}_{j}$-linear map $_{j}\varphi_{i}: V_{i}\otimes_{\mathbb{F}_{i}}{}_{i}M_{j}\to V_{j}$ for each arrow $\alpha:i\to j$.

Let $V=(V_{i}, \varphi_{\alpha})$ and $W=(W_{i}, \psi_{\alpha})$ be two finite-dimensional $K$-representations of $(Q, {\bf d})$. A morphism $(f_{i})_{i\in I}$ from $V$ to $W$ consists of an $\mathbb{F}_{i}$-linear map $f_{i}:V_{i}\to W_{i}$ for each $i\in I$ such that the following diagram is commutative for each arrow $\alpha: i\to j$:
\[\xymatrix{V_{i}\otimes_{\mathbb{F}_{i}} {}_{i}M_{j}\ar[d]^{f_{i}\otimes \id}\ar[r]^{~~_{j}\varphi_{i}}&V_{j}\ar[d]^{f_{j}}\\
W_{i}\otimes_{\mathbb{F}_{i}} {}_{i}M_{j}\ar[r]^{~~_{j}\psi_{i}}&W_{j}.
}
\]

Let $\rep(Q, {\bf d})$ be the category of all the finite-dimensional $K$-representations of $(Q, {\bf d})$. It is a hereditary abelian category over $K$. It is also well-known that there is a finite-dimensional hereditary $K$-algebra $\Lambda$ such that the category $\mod \Lambda$ of finite-dimensional $\Lambda$-modules is equivalent to $\rep(Q, {\bf d})$. For an $M\in \rep(Q, {\bf d})$, we will denote by $\add M$ the subcategory of $\rep(Q, {\bf d})$ consisting of objects which are finite direct sum of direct summands of $M$.

 For a given valued quiver $(Q, {\bf d})$, we associate a matrix $C_{(Q, {\bf d})}=(c_{ij})\in M_{I\times I}(\Z)$ to $(Q, {\bf d})$ by defining $c_{ij}=\begin{cases}2&i=j;\\ -d_{ij}& i\neq j.
\end{cases}$. It is clear that $C_{(Q, {\bf d})}$ is a symmetrizable Cartan matrix with symmetrizer ${\bf f}=\{f_{i}, ~i\in I\}$. Moreover,  the valued quiver $(Q, {\bf d})$ is of representation-finite type if and only if the Cartan matrix $C_{(Q, {\bf d})}$ is of finite type.
 We refer to ~\cite{DR} for more details.

\subsection{Ringel-Hall algebras}
Let $K$ be a finite field with $|K|=q$. Set $v^2=q$ in this subsection.
For a given Cartan matrix $C\in M_{I\times I}(\Z)$, we may associate a valued quiver $(Q, {\bf d})$ without oriented cycles to $C$ such that  $C_{(Q, {\bf d})}=C$. Let $S_i,~ i\in I$ be the pairwise non-isomorphic simple representations of $(Q, {\bf d})$.
Denote by $\go(Q, {\bf d})$ the Grothendieck group of $\rep(Q, {\bf d})$. The Euler bilinear form $\langle-,-\rangle$ on $\go(Q, {\bf d})$ is given by
\[\langle |M|, |N|\rangle:=\dim_K\Hom(M,N)-\dim_K\Ext^1(M,N),
\]
where $|M|$ stands for the image of $M\in \rep(Q, {\bf d})$ in $\go(Q, {\bf d})$. It is well-known that $|S_i|,~i\in I$ form a $\Z$-basis of $\go(Q, {\bf d})$ and we may identify $\go(Q, {\bf d})$ with the root lattice $L$ associated to $C$. Under this identification, $|S_i|,~i\in I$ identify with the simple roots $\alpha_i,~i\in I$ and the symmetric bilinear form $(-,-):L\times L\to \Z$ identifies with the symmetrization of the Euler form $\langle-,-\rangle$. In other words, $(\alpha,\beta)=\langle \alpha, \beta\rangle+\langle\beta,\alpha\rangle$ for $\alpha,\beta\in L$.

 For any $M, N, L\in \rep(Q, {\bf d})$, the {\it Hall number} $F_{M,N}^L$ counts the subrepresentations $X$ of $L$ such that $X\cong N$ and $L/X\cong M$. In general, for any $L, X_1,\cdots, X_t\in \rep(Q, {\bf d})$, then number $F_{X_1,\cdots, X_t}^L$ counts the  filtrations \[0=L_{t}\subset L_{t-1}\subset \cdots \subset L_1\subset L_0=L
 \]
 such that $L_{k-1}/L_k\cong X_k,~k=1,\cdots, t-1$.
For a representation $X\in \rep(Q, {\bf d})$, we denote by $[X]$ the isoclass of $X$.
The Ringel-Hall algebra $\mathcal{H}(Q)$ of $\rep(Q, {\bf d})$ is an associative algebra with unit $[0]$ (we also denote it by $\eta:\mathbb{Q}(v)\to \mathcal{H}(Q)$) whose underlying space is the $\Q(v)$-vector space spanned by the isomorphism classes of representations of $(Q, {\bf d})$ and the  multiplication $\mu$ is given by
\[[M]*[N]=\sum_{[L]}v^{\langle |M|,|N|\rangle}F_{M,N}^L[L],
\]
where $M,N,L\in \rep(Q,{\bf d})$. It is clear that $\mathcal{H}(Q)=\bigoplus\limits_{\alpha\in \N I}\mathcal{H}(Q)_\alpha$ is $\N I$-graded, where $\mathcal{H}(Q)_\alpha=\opname{span}\{[M]~|~M\in \rep(Q, {\bf d}), |M|=\alpha\}$.

Let  $(-,-)_{\mathcal{H}(Q)}:\mathcal{H}(Q)\times \mathcal{H}(Q)\to \Q(v)$ be the non-degenerate symmetric bilinear form defined by $([M], [N])_{\mathcal{H}(Q)}=\frac{1}{|\Aut(M)|}\delta_{M,N}$.
We endow the tensor product $\mathcal{H}(Q)\otimes \mathcal{H}(Q)$ with the twisted multiplication
\[(U_1\otimes V_1)*(U_2\otimes V_2):=v^{(|V_1|, |U_2|)}U_1*U_2\otimes V_1*V_2
\]
and define a $\Q(v)$-linear map $\Delta: \mathcal{H}(Q)\to \mathcal{H}(Q)\otimes \mathcal{H}(Q)$ as
\[\Delta([L])=\sum_{[M],[N]}v^{\langle |M|, |N|\rangle}\frac{|\Aut(M)||\Aut(N)|}{|\Aut(L)|}F_{M,N}^L[M]\otimes [N].
\]
Finally, define a $\Q(v)$-linear map $\epsilon:\mathcal{H}(Q)\to \Q(v)$ by $\epsilon([M])=\begin{cases}0& \text{if $M\not\cong 0$};\\1& \text{if $M\cong 0$}.\end{cases}$
\begin{theorem}~\label{t:Hallalgebra}
\begin{itemize}
\item[(1)](Ringel~\cite{Ringel}) The assignment $E_j\mapsto [S_j]$ for $j\in I$ extends to a homomorphism of algebras $\Psi: U_+\to \mathcal{H}(Q)$.  Moreover, if $(Q, {\bf d})$ is of representation-finite type, the homomorphism $\Psi$ is an isomorphism;
\item[(2)](Green~\cite{Green}) The algebra $(\mathcal{H}(Q),\mu, \eta, \Delta, \epsilon)$ is an $\N I$-graded bialgebra and the homomorphism $\Psi$ is an injective homomorphism of bialgebras.
\end{itemize}
\end{theorem}

Note that the multiplication $\mu$ and the comultiplication $\Delta$ form an adjoint pair with respect to the non-degenerate bilinear form $(-,-)_{\mathcal{H}(Q)}$.
Let $\mathcal{H}^*(Q)$ be the graded dual of $\mathcal{H}(Q)$ and $\{\delta_{[M]}~|~M\in\rep(Q, {\bf d})\}$ the basis of $\mathcal{H}^*(Q)$ dual to $\{[M]~|~M\in \rep(Q, {\bf d})\}$. Then $(\mathcal{H}(Q), \mu,\eta, \Delta, \epsilon)$ is isomorphic to $(\mathcal{H}^*(Q), \Delta^*, \epsilon^*, \mu^*,\eta^*)$ as bialgebras. Moreover, the linear map $U_+\to \mathcal{H}^*(Q)$ defined by $E_j\mapsto \delta_{[S_j]},~j\in I$ is also an injective homomorphism of algebras.

\subsection{Feigin-type maps}~\label{ss:Feigin-type-maps}
Recall that for a given Cartan matrix $C\in M_{I\times I}(\Z)$, we have three $\N I$-graded bialgebras: the quantum enveloping algebra $U_+$, the quantum shuffle algebra $\mathcal{F}^*$ and the dual Ringel-Hall algebra $\mathcal{H}^*(Q)$. For a fixed word ${\bf w}\in W$, there is a quantum polynomial algebra $\mathbb{P}_{\bf w}$. There are various homomorphisms  of algebras between these four algebras which have been studied extensively due to its connection to quantum groups, cluster algebras and motivic Donaldson-Thomas invariants ({\it cf.} ~\cite{BR, Leclerc,Reineke, Reineke10,Rupel}). In particular, we have the following.

\begin{theorem}~\label{t:Feigin-type maps}
Fix a word ${\bf w}=(i_1i_2\cdots i_m)\in W$.
\begin{itemize}
\item[(1)](\text{Feigin}) The linear map
 \begin{eqnarray*}
\mathbf{F_w}: &U_+&\to \mathbb{P}_{\bf w}\\
&E_j&\mapsto \sum_{1\leq k\leq m, i_k=j}t_k
\end{eqnarray*}
is a homomorphism of algebras.
\item[(2)](Berenstein-Rupel~\cite{BR}) The linear map
\begin{eqnarray*}
\int_{\bf w}: &\mathcal{H}^*(Q)&\to \mathbb{P}_{\bf w}\\
&\delta_{[M]}&\mapsto \sum_{\underline{a}\in \N^m}v^{\sum_{k<l}a_ka_l\langle \alpha_{i_k}, \alpha_{i_l}\rangle+\sum_{k=1}^mf_{i_k}\frac{a_k(a_k-1)}{2}}F_{S_{i_1}^{a_1},\cdots, S_{i_m}^{a_m}}^Mt^{\underline{a}}
\end{eqnarray*}
is a homomorphism of algebras, where $S_{i_k}^{a_k}$ is the direct sum of $a_k$ copies of $S_{i_k}$.
\item[(3)](Rupel~\cite{Rupel})
The linear map
\begin{eqnarray*}
\mathbf{\Omega}: &\mathcal{H}^*(Q)&\to \mathcal{F}^*\\
&\delta_{[M]}&\mapsto \sum_{{\bf u}=(j_1j_2\cdots j_r)\in W}v^{\sum_{k<l}\langle \alpha_{j_k},\alpha_{j_l}\rangle}F^M_{S_{j_1},\cdots, S_{j_r}}y_{\bf u}
\end{eqnarray*}
is a homomorphism of bialgebra.
\item[(4)](Rupel~\cite{Rupel}) The linear map
\begin{eqnarray*}
\mathbf{S}_{\bf w}:&\mathcal{F}^*&\to \mathbb{P}_{\bf w}\\
&y_{\bf u}&\mapsto \sum_{\underline{a}\in \N^m, {\bf u}=(i_1^{a_1}i_2^{a_2}\cdots i_m^{a_m})}\prod_{k=1}^m v_{i_k}^{-\frac{a_k(a_k-1)}{2}}\frac{1}{[a_1]_{i_1}^!\cdots [a_m]_{i_m}^!}t^{\underline{a}}
\end{eqnarray*}
 is a homomorphism of algebra. Moreover, $\int_{\bf w}=\mathbf{S}_{\bf w}\circ \mathbf{\Omega}$.
\end{itemize}
\end{theorem}

The morphism $\mathbf{F_w}$ is called  {\it Feigin's map} of type ${\bf w}$, which was first proposed by B. Feigin as a tool for studying the skew-field of fractions of $(U_+)^*$.   The morphism $\int_{\bf w}$ is called {\it generalized Feigin's homomorphism} in~\cite{Rupel}. For special choice of words ${\bf w}$, these morphisms were used  to study the associated quantum cluster algebras~\cite{BR}.  The morphism $\mathbf{\Omega}$ is called {\it quantum shuffle character} in ~\cite{Rupel} which establishes a connection between representations of the valued quiver $(Q, {\bf d})$ and the irreducible representations of the associated KLR algebra.

The rest of this section is devoted to give an alternative proof for the aforementioned  homomorphisms. Here we emphasize the graded dual approach which not only simplifies a lot of calculations in ~\cite{Rupel}, but  also leads to a representation-theoretic understanding of the injectivity of Feigin's map for certain special words ${\bf w}$ in Section~\ref{s:Feiginmap}.

\subsection{The proof}
\begin{lemma}~\label{l:Green}
The linear map
\begin{eqnarray*}
\Phi: &\mathcal{F}&\to \mathcal{H}(Q)\\
&x_i&\mapsto [S_i]
\end{eqnarray*}
is a homomorphism of bialgebras.
\end{lemma}
\begin{proof}
Recall that $\mathcal{F}$ is the free associative algebra generated by $x_i,~\in I$, it is clear that $\Phi$ is a homomorphism of algebras. It remains to show that the following diagrams are commutative
\[\xymatrix{\mathcal{F}\ar[d]^{\Delta_{\mathcal{F}}}\ar[r]^{\Phi}&\mathcal{H}(Q)\ar[d]^{\Delta_{\mathcal{H}(Q)}}\\ \mathcal{F}\otimes \mathcal{F}\ar[r]^{\Phi\otimes \Phi}&\mathcal{H}(Q)\otimes\mathcal{H}(Q)}~~\text{and}~
 \xymatrix{\mathcal{F}\ar[d]^{\Phi}\ar[r]^{\epsilon_{\mathcal{F}}}&\Q(v).\\ \mathcal{H}(Q)\ar[ur]_{\epsilon_{\mathcal{H}(Q)}}}
\]
By Theorem~\ref{t:Hallalgebra}~$(2)$, we know that $\Delta_{\mathcal{H}(Q)}$ is a homomorphism of algebras, hence $\Delta_{\mathcal{H}(Q)}\circ \Phi$ and $(\Phi\otimes \Phi)\circ \Delta_{\mathcal{F}}$ are homomorphisms of algebras. It suffices to prove that $\Delta_{\mathcal{H}(Q)}\circ \Phi(x_i)=(\Phi\otimes \Phi)\circ \Delta_{\mathcal{F}}(x_i)$ for any $i\in I$, which is evident, since $[S_i]$ is primitive in $\mathcal{H}(Q)$. The commutativity of the other diagram follows similarly.
\end{proof}

\begin{lemma}~\label{l:coalgebra}
Let ${\bf w}=(i_1i_2\cdots i_m)\in W$.
The linear map
\begin{eqnarray*}
\mathbf{T}_{\bf w}: &\mathbb{P}_{\bf w}^*&\to \mathcal{F}\\
&t_{\underline{a}}&\mapsto \prod_{k=1}^m v_{i_k}^{\frac{-a_k(a_k-1)}{2}}\frac{1}{[a_1]_{i_1}^!\cdots [a_m]_{i_m}^!}x_{i_1}^{a_1}\cdots x_{i_m}^{a_m}
\end{eqnarray*}
is a homomorphism of coalgebras.
\end{lemma}
\begin{proof}
We are going to show that the following diagram is  commutative
\[\xymatrix{\mathbb{P}_{\bf w}^*\ar[d]_{\Delta_{\mathbb{P}_{\bf w}^*}}\ar[r]^{\bf T_{ w}}&\mathcal{F}\ar[d]^{\Delta_{\mathcal{F}}}\\ \mathbb{P}_{\bf w}^*\otimes \mathbb{P}_{\bf w}^*\ar[r]^{\mathbf{T}_{\bf w}\otimes \mathbf{T}_{\bf w}}&\mathcal{F}\otimes \mathcal{F}.}
\]
For $\underline{a}\in \N^m$, set $z_{\bf w}(\underline{a}):=\prod_{k=1}^m v_{i_k}^{\frac{-a_k(a_k-1)}{2}}\frac{1}{[a_1]_{i_1}^!\cdots [a_m]_{i_m}^!}$ and $x_{\bf w}^{\underline{a}}:=x_{i_1}^{a_1}\cdots x_{i_m}^{a_m}$.  We have
\begin{eqnarray*}
(\mathbf{T}_{\bf w}\otimes \mathbf{T}_{\bf w})\circ \Delta_{\mathbb{P}_{\bf w}^*}(t_{\underline{a}})&=&\mathbf{T}_{\bf w}\otimes \mathbf{T}_{\bf w}(\sum_{\underline{b}+\underline{c}=\underline{a}}v^{\sum_{k<l}c_kb_l(\alpha_{i_k},\alpha_{i_l})}~t_{\underline{b}}\otimes t_{\underline{c}})\\
&=&\sum_{\underline{b}+\underline{c}=\underline{a}}v^{\sum_{k<l}c_kb_l(\alpha_{i_k},\alpha_{i_l})}z_{\bf w}(\underline{b})z_{\bf w}(\underline{c})x_{\bf w}^{\underline{b}}\otimes x_{\bf w}^{\underline{c}}.
\end{eqnarray*}
On the other hand, by Lemma~\ref{l:computation}, we have
\begin{eqnarray*}
\Delta_{\mathcal{F}}\circ \mathbf{T}_{\bf w}(t_{\underline{a}})&=&z_{\bf w}(\underline{a})\Delta_{\mathcal{F}}(x_{\bf w}^{\underline{a}})\\
&=&z_{\bf w}(\underline{a})\sum_{\underline{b}+\underline{c}=\underline{a}}\prod_{k=1}^m v_{i_k}^{b_kc_k}v^{\sum_{k<l}c_kb_l(\alpha_{i_k},\alpha_{i_l})}\left[\begin{array}{c}a_1\\ c_1\end{array}\right]_{i_1}\cdots \left[\begin{array}{c}a_m\\ c_m\end{array}\right]_{i_m}x_{\bf w}^{\underline{b}}\otimes x_{\bf w}^{\underline{c}}
\end{eqnarray*}
Comparing the coefficients of $x_{\bf w}^{\underline{b}}\otimes x_{\bf w}^{\underline{c}}$, we deduce that
\[\Delta_{\mathcal{F}}\circ \mathbf{T}_{\bf w}(t_{\underline{a}})=(\mathbf{T}_{\bf w}\otimes \mathbf{T}_{\bf w})\circ \Delta_{\mathbb{P}_{\bf w}^*}(t_{\underline{a}}).
\]
\end{proof}

Now we are in a position to prove Theorem~\ref{t:Feigin-type maps}.

\noindent {\it Proof of Theorem~\ref{t:Feigin-type maps}:}

We first prove Theorem~\ref{t:Feigin-type maps}~ (3) and (4).
Let $\Phi^*:\mathcal{H}^*(Q)\to \mathcal{F}^*$ be the adjoint of $\Phi$ and $\mathbf{T}_{\bf w}^*:\mathcal{F}^*\to \mathbb{P}_{\bf w}$ the adjoint of $\mathbf{T}_{\bf w}$. Hence $\Phi^*$ is a homomorphism of bialgebras and $\mathbf{T}_{\bf w}^*$ is a homomorphism of algebras.
We are going to show that $\mathbf{\Omega}=\Phi^*$ and $\mathbf{S}_{\bf w}=\mathbf{T}_{\bf w}^*$.

Let $\langle-,-\rangle_{\mathcal{F}}:\mathcal{F}^*\times \mathcal{F}\to \Q(v)$ and $\langle-,-\rangle_{\mathcal{H}(Q)}:\mathcal{H}^*(Q)\times \mathcal{H}(Q)\to \Q(v)$ be the evaluation pairings.  For any $\delta_{[M]}\in \mathcal{H}^*(Q)$ and ${\bf u}=(j_1j_2\cdots j_r)\in W$, we have
\begin{eqnarray*}
\langle \Phi^*(\delta_{[M]}), x_{\bf u}\rangle_{\mathcal{F}}&=&\langle\delta_{[M]}, \Phi(x_{\bf u})\rangle_{\mathcal{H}(Q)}\\
&=&\langle\delta_{[M]}, [S_{j_1}]*[S_{j_2}]*\cdots *[S_{j_r}]\rangle_{\mathcal{H}(Q)}\\
&=&\sum_{[L]}v^{\sum_{k<l}\langle\alpha_{j_k}, \alpha_{j_l}\rangle}F_{S_{j_1},\cdots ,S_{j_r}}^L\langle \delta_{[M]}, [L]\rangle_{\mathcal{H}(Q)}.
\end{eqnarray*}
Therefore \[\Phi^*(\delta_{[M]})=\sum_{{\bf u}=(j_1j_2\cdots j_r)\in W}v^{\sum_{k<l}\langle \alpha_{j_k},\alpha_{j_l}\rangle}F_{S_{j_1},\cdots,S_{j_r}}^My_{\bf u}=\mathbf{\Omega}(\delta_{[M]}).\]

Let $\langle-,-\rangle_{\mathbb{P}_{\bf w}}:\mathbb{P}_{\bf w}^*\times \mathbb{P}_{\bf w}\to \Q(v)$ be the evaluation pairing. For any $\underline{a}\in \N^m$ and ${\bf u}\in W$, we have
\begin{eqnarray*}
\langle t_{\underline{a}}, \mathbf{T}_{\bf w}^*(y_{\bf u})\rangle_{\mathbb{P}_{\bf w}}&=&\langle \mathbf{T}_{\bf w}(t_{\underline{a}}), y_{\bf u}\rangle_{\mathcal{F}}\\
&=& z_{\bf w}(\underline{a})\langle x_{\bf w}^{\underline{a}}, y_{\bf u}\rangle_{\mathcal{F}}.
\end{eqnarray*}
Note that $\langle x_{\bf w}^{\underline{a}}, y_{\bf u}\rangle_{\mathcal{F}}\neq 0$ if and only if ${\bf u}=(i_1^{a_1},i_2^{a_2},\cdots, i_m^{a_m})$. Thus we have
\[\mathbf{T}_{\bf w}^*(y_{\bf u})=\sum_{\underline{a}\in \N^m, {\bf u}=(i_1^{a_1}i_2^{a_2}\cdots i_m^{a_m})}z_{\bf w}(\underline{a})t^{\underline{a}}=\mathbf{S}_{\bf w}(y_{\bf u}).
\]

To prove Theorem~\ref{t:Feigin-type maps} (2),
we notice  that $\Phi\circ \mathbf{T}_{\bf w}:\mathbb{P}_{\bf w}^*\to \mathcal{H}(Q)$ is a homomorphism of coalgebras and hence its adjoint $(\Phi\circ \mathbf{T}_{\bf w})^*:\mathcal{H}^*(Q)\to \mathbb{P}_{\bf w}$ is a homomorphism of algebras.
We claim that $\int_{\bf w}=(\Phi\circ \mathbf{T}_{\bf w})^*$. For $\delta_{[M]}\in \mathcal{H}^*(Q)$ and $\underline{a}\in \N^m$, we have
\begin{eqnarray*}
\langle (\Phi\circ \mathbf{T}_{\bf w})^*(\delta_{[M]}), t_{\underline{a}}\rangle_{\mathbb{P}_{\bf w}}
&=&\langle \delta_{[M]}, (\Phi\circ \mathbf{T}_{\bf w})(t_{\underline{a}})\rangle_{\mathcal{H}(Q)}\\
&=&z_{\bf w}(\underline{a})\langle \delta_{[M]}, \Phi(x_{\bf w}^{\underline{a}})\rangle_{\mathcal{H}(Q)}.
\end{eqnarray*}
On the other hand,
\begin{eqnarray*}
\Phi(x_{\bf w}^{\underline{a}})&=&\Phi(x_{i_1}^{a_1})\cdots \Phi(x_{i_m}^{a_m})\\
&=&v^{\sum_{k=1}^m\frac{a_k(a_k-1)}{2}\langle \alpha_{i_k}, \alpha_{i_k}\rangle}[S_{i_1}^{a_1}]\cdots [S_{i_m}^{a_m}]\\
&=&v^{\sum_{k=1}^m\frac{a_k(a_k-1)}{2}\langle \alpha_{i_k}, \alpha_{i_k}\rangle}\sum_{[L]}v^{\sum_{k<l}a_ka_l\langle \alpha_{i_k},\alpha_{i_l}\rangle}F^L_{S_{i_1}^{a_1},\cdots, S_{i_m}^{a_m}}[L].
\end{eqnarray*}
Note that $\langle \alpha_{i_k},\alpha_{i_k}\rangle=f_{i_k}$, hence we have
\[(\Phi\circ T_{\bf w})^*(\delta_{[M]})=\sum_{\underline{a}\in \N^m}v^{\sum_{k<l}a_ka_l\langle \alpha_{i_k}, \alpha_{i_l}\rangle+\sum_{k=1}^mf_{i_k}\frac{a_k(a_k-1)}{2}}F_{S_{i_1}^{a_1},\cdots, S_{i_m}^{a_m}}^Mt^{\underline{a}}=\int_{\bf w}(\delta_{[M]}).
\]
Moreover, $\int_{\bf w}=(\Phi\circ \mathbf{T}_{\bf w})^*=\mathbf{T}_{\bf w}^*\circ \Phi^*=\mathbf{S}_{\bf w}\circ \mathbf{\Omega}$.

Finally, recall that we have a homomorphism of bialgebras $\pi:\mathcal{F}\to U_+^* $ by $\pi(x_i)=e_i,~i\in I$. In order to prove Theorem~\ref{t:Feigin-type maps} (1), it suffices to prove that $\mathbf{F_w}=(\pi\circ \mathbf{T}_{\bf w})^*$.

Let $\langle-,-\rangle_{U_+}: U_+\times U_+^*\to \Q(v)$ be the evaluation pairing.
For any $\underline{a}\in \N^m$,
\begin{eqnarray*}
\langle (\pi\circ \mathbf{T}_{\bf w})^*(E_j),t_{\underline{a}}\rangle_{\mathbb{P}_{\bf w}}&=&\langle E_j,(\pi\circ \mathbf{T}_{\bf w})(t_{\underline{a}})\rangle_{U_+}\\
&=&z_{\bf w}(\underline{a})\langle E_j, e_{i_1}^{a_1}\cdots e_{i_m}^{a_m}\rangle_{U_+}
\end{eqnarray*}
Note that $\langle E_j, e_{i_1}^{a_1}\cdots e_{i_m}^{a_m}\rangle_{U_+}\neq 0$ if and only if $\exists~ 1\leq k\leq m$ such that $a_k=1, a_l=0, l\neq k$ and $i_k=j$. In this case, $z_{\bf w}(\underline{a})=1$ and hence
\[(\pi\circ \mathbf{T}_{\bf w})^*(E_j)=\sum_{1\leq k\leq m, i_k=j}t_k=\mathbf{F_w}(E_j).
\]
\begin{remark}
Recall that we have a homomorphism $\Psi:U_+\to \mathcal{H}^*(Q)$, one can show that $\mathbf{F_w}=\int_{\bf w}\circ \Psi$ similarly.
\end{remark}

\subsection{Relation to Rupel's notations}
In order to compare the homomorphisms defined above with the ones of~\cite{Rupel}, we need to consider $\mathbb{P}_{\bf w}, \mathcal{H}^*(Q)$ and $\mathcal{F}^*$ over the field $\Q(v^{1/2})$.
In ~\cite{Rupel}, Rupel introduced different basis for $\mathbb{P}_{\bf w}$ and $\mathcal{H}^*(Q)$.

Recall that we have fixed a word ${\bf w}=(i_1i_2\cdots i_m)\in W$.  For any $\underline{a}\in \N^m$, set
\[t_{\bf w}^{\underline{a}}:=v^{\frac{1}{2}\sum_{k<l}a_ka_l(\alpha_{i_k},\alpha_{i_l})}t^{\underline{a}}\in \mathbb{P}_{\bf w}.
\]
It is clear that $\{t_{\bf w}^{\underline{a}}~|~\underline{a}\in \N^m\}$ also forms a $\Q(v^{1/2})$-basis of $\mathbb{P}_{\bf w}$.

For any representation $M\in \rep(Q, {\bf d})$ with $\dimv M=(m_i, i\in I)$, set
\[[M]^*:=v^{-\frac{1}{2}\langle |M|, |M|\rangle+\frac{1}{2}\sum_{i\in I}f_im_i}\delta_{[M]}.
\]
Then $\{[M]^*~|~M\in\rep(Q, {\bf d})\}$ also forms a basis of $\mathcal{H}^*(Q)$. By applying the homomorphism $\int_{\bf w}$ to $[M]^*$ we have
\begin{eqnarray*}
\int_{\bf w}([M]^*)&=&v^{-\frac{1}{2}\langle |M|,|M|\rangle+\frac{1}{2}\sum_{i\in I}f_im_i}\int_{\bf w}(\delta_{[M]})\\
&=&\sum_{\underline{a}\in \N^m}v^{-\frac{1}{2}\langle |M|,|M|\rangle+\frac{1}{2}\sum_{i\in I}d_im_i+\sum_{k<l}a_ka_l\langle\alpha_{i_k},\alpha_{i_l}\rangle+\sum_{k=1}^mf_{i_k}\frac{a_k(a_k-1)}{2}}F^M_{S_{i_1}^{a_1},\cdots, S_{i_m}^{a_m}}t^{\underline{a}}
\end{eqnarray*}
Note that $F^M_{S_{i_1}^{a_1},\cdots,S_{i_m}^{a_m}}\neq 0$ implies that $|M|=\sum_{k=1}^ma_k\alpha_{i_k}$ and hence $\sum_{i\in I}f_im_i=\sum_{k=1}^mf_{i_k}a_k$. We have
\[\frac{1}{2}\langle|M|,|M|\rangle+\frac{1}{2}\sum_{i\in I}d_im_i+\sum_{k=1}^m\frac{a_k(a_k-1)}{2}=-\frac{1}{2}\sum_{k<l}a_ka_l\langle\alpha_{i_k},\alpha_{i_l}\rangle-\frac{1}{2}\sum_{k<l}a_ka_l\langle\alpha_{i_l},\alpha_{i_k}\rangle.
\]
Therefore
\begin{eqnarray*}
\int_{\bf w}([M]^*)&=&\sum_{\underline{a}\in \N^m}v^{\frac{1}{2}\sum_{k<l}a_ka_l\langle \alpha_{i_k},\alpha_{i_l}\rangle-\frac{1}{2}\sum_{k<l}a_ka_l\langle \alpha_{i_l},\alpha_{i_k}\rangle}F^M_{S_{i_1}^{a_1},\cdots, S_{i_m}^{a_m}}t^{\underline{a}}\\
&=&\sum_{\underline{a}\in \N^m}v^{\frac{1}{2}\sum_{k<l}a_ka_l\langle \alpha_{i_k},\alpha_{i_l}\rangle-\frac{1}{2}\sum_{k<l}a_ka_l\langle \alpha_{i_l},\alpha_{i_k}\rangle-\frac{1}{2}\sum_{k<l}a_ka_l(\alpha_{i_k},\alpha_{i_l})}F^M_{S_{i_1}^{a_1},\cdots, S_{i_m}^{a_m}}t_{\bf w}^{\underline{a}}\\
&=&\sum_{\underline{a}\in \N^m}v^{-\sum_{k<l}a_ka_l\langle\alpha_{i_l},\alpha_{i_k}\rangle}F^M_{S_{i_1}^{a_1},\cdots, S_{i_m}^{a_m}}t_{\bf w}^{\underline{a}}.
\end{eqnarray*}
This verifies that $\int_{\bf w}$ coincides with the one of Berenstein-Rupel~\cite{BR}.

For the quantum shuffle algebra $\mathcal{F}^*$, Rupel also considered a new basis.  More precisely,  for any ${\bf u}=(j_1j_2\cdots j_r)\in W$, set
\[z_{\bf u}:=v^{-\frac{1}{2}\sum_{k<l}(\alpha_{j_k},\alpha_{j_l})}y_{\bf u}.
\]
Then $\{z_{\bf u}~|~{\bf u}\in W\}$ is a basis of $\mathcal{F}^*$. For ${\bf u_1}=(j_1j_2\cdots j_r)$ and ${\bf u_2}=(j_{r+1}j_{r+2}\cdots j_{r+s})$, one can verify that
\[z_{\bf u_1} z_{\bf u_2}=\sum_{\sigma\in \Sigma_{r+s}}v^{\zeta(\sigma,{\bf u_1},{\bf w_2})}z_{\sigma({\bf u_1},{\bf u_2})},
\]
where
\[ \zeta(\sigma, {\bf u_1}, {\bf u_2})=\frac{1}{2}\sum_{k\leq r, l>r, \sigma_k^{-1}>\sigma_l^{-1}}(\alpha_{j_k},\alpha_{j_l})-\frac{1}{2}\sum_{k\leq r, l>r, \sigma_k^{-1}<\sigma_l^{-1}}(\alpha_{j_k},\alpha_{j_l}).
\]
Moreover, for $ {\bf u}\in W$,
\[\mu^*(z_{\bf u})=\sum_{({\bf u_1}, {\bf u_2})={\bf u}}v^{\frac{1}{2}(|z_{\bf u_1}|, |z_{\bf u_2}|)}z_{\bf u_1}\otimes z_{\bf u_2}.
\]

Similar to the homomorphism $\int_{\bf w}$, one can verify that
\begin{eqnarray*}
&&\mathbf{\Omega}([M]^*)=\sum_{{\bf u}=(j_1\cdots j_r)\in W}v^{-\sum_{k<l}\langle \alpha_{j_l},\alpha_{j_k}\rangle}F^M_{S_{j_1},\cdots S_{j_r}}z_{\bf u};\\
&&\mathbf{S}_{\bf w}(z_{\bf u})=\sum_{\underline{a}\in \N^m, (i_1^{a_1}\cdots i_m^{a_m})={\bf u}}\frac{1}{[a_1]_{i_1}^!\cdots [a_m]_{i_m}^!}t_{\bf w}^{\underline{a}}.
\end{eqnarray*}

\section{Feigin's map for Cartan matrices of finite type}~\label{s:Feiginmap}

\subsection{Weyl group of Cartan matrix of finite type}~\label{s:Weylgroup}
Let $C\in M_{I\times I}(\Z)$ be a Cartan matrix of finite type and $\mathfrak{g}$ the associated complex semisimple  Lie algebra.  Recall that $L$ is the root lattice of $\mathfrak{g}$ with symmetric bilinear form $(-,-)$ and $\Pi=\{\alpha_i, i\in I\}$ is the set of simple roots. Denote by $\Phi^+$ the set of positive roots of $\mathfrak{g}$.

For each $\alpha\in \Phi^+$, we define  the reflection
\begin{eqnarray*}
r_{\alpha}:&L&\to L\\
&\beta&\mapsto \beta-\frac{2(\beta,\alpha)}{(\alpha,\alpha)}\alpha,
\end{eqnarray*}
which is an automorphism of $L$. Let $W(C)$ be the Weyl group of $\mathfrak{g}$ which is the subgroup of the automorphism group $\Aut(L)$ generated by the simple reflections $r_i:=r_{\alpha_i}, i\in I$. Since $C$ is of finite type, we know that $W(C)$ is a finite group. Each element $\omega\in W(C)$  can be written as a product of the simple reflections, say $\omega=r_{i_1}\cdots r_{i_s}$. The {\it length $l(\omega)$} of $\omega$ is defined to be the smallest $s$ for which such an expression exists. In this case, we call $\omega=r_{i_1}\cdots r_{i_s}$ a {\it reduced expression} of $\omega$. It is well-known that there is a unique element $\omega_0\in W(C)$ which has the maximal length $l(\omega_0)=|\Phi^+|=:\nu$. The element $\omega_0$ is also characterized by the following property:
\[\text{if $\omega\in W(C)$ such that $\omega(\Pi)=-\Pi$, then we have $\omega=\omega_0$.}
\]
It is also well-known how to write down reduced expressions for $\omega_0$ and the reduced expressions of $\omega_0$ have played an important role in the study of canonical bases of the corresponding quantum enveloping algebra~\cite{Lusztig93, Lusztig94}.
We denote by $\mathscr{X}$ the set of words associated to reduced expressions of $\omega_0$. In particular,  $\mathscr{X}$ consists all of  the words $(i_1\cdots i_\nu)$ such that $\omega_0=r_{i_1}\cdots r_{i_\nu}$.

Recall that a {\it cycle} in a category $\mathcal{C}$ consists of indecomposable objects $X_1, \cdots, X_t\in \mathcal{C}$ with $t\geq 2$ and non-isomorphic and non-zero morphisms $g_i: X_i\to X_{i+1}, 1\leq i<t$ such that $X_1\cong X_t$.
 Let $(Q, {\bf d})$ be a valued quiver  associated to the Cartan matrix $C$. We identify the vertex set $Q_0$ with $I$.
Let $\ind (Q, {\bf d})$ be a representative set of isoclass of indecomposable representations of $(Q, {\bf d})$. Denote by $\tau$  the Auslander-Reiten (AR) translation of $(Q, {\bf d})$ and $\tau^{-1}$ its inverse.  Let $P_i$ and $ I_i$ be the indecomposable projective representations and injective representations respectively associated to the vertex $i\in I$.  We define $\mathcal{P}=\{P_i~|~i\in I\}$ and $\mathcal{I}=\{I_i~|~i\in I\}$.
We list the following well-known facts for the valued quiver $(Q,{\bf d})$ associated to the Cartan matrix $C$ of finite type ({\it cf.}~\cite{DR}):
\begin{enumerate}
\item[$\bullet$] the valued quiver $(Q, {\bf d})$ is {\it representation-directed}, that is, there is  no cycle in $\rep(Q, {\bf d})$.
\item[$\bullet$] the dimension vector $\dimv$ induces a bijection between $\ind(Q, {\bf d})$ and the positive roots $\Phi^+$ of the complex semisimple Lie algebra $\mathfrak{g}$ associated to $C$.  In particular, $|\ind(Q, {\bf d})|=|\Phi^+|=\nu$.
\item[$\bullet$] for an $M\in \ind(Q, {\bf d})$, $\tau M=0$ if and only if $M\in \mathcal{P}$;  $\tau^{-1}M=0$ if and only if $M\in \mathcal{I}$. Moreover, the AR translation $\tau$ restricts to a bijection between $\ind(Q, {\bf d})\backslash \mathcal{P}$ and $\ind(Q, {\bf d})\backslash \mathcal{I}$.
\item[$\bullet$] for each $M\in \ind (Q, {\bf d})$, there exists unique $i_M\in I$ and $k_M\in \N$ such that $M\cong \tau^{k_M}I_{i_M}$. Thus we have  a well-defined map
 $\theta_{\tau}:\ind(Q, {\bf d})\to I$ by $\theta_{\tau}(M)=i_M$.
 \item[$\bullet$] set $P=\bigoplus\limits_{i\in I}P_i$, for any $k<l$, one has $\Hom(\tau^{-l}P, \tau^{-k}P)=0$.
\end{enumerate}

 An {\it enumeration} $\mathfrak{e}$ of $(Q, {\bf d})$ is a bijection $\mathfrak{e}:\{1,2,\cdots, \nu\}\to \ind(Q, {\bf d})$ such that
\[\Hom(\mathfrak{e}(i),\mathfrak{e}(j))=0=\Ext^1(\mathfrak{e}(j), \mathfrak{e}(i))~\text{for all $1\leq j<i\leq \nu$.}
 \]
 The valued quiver $(Q, {\bf d})$ is representation-directed implies that there exist enumerations for $(Q, {\bf d})$.
   A word $(i_1i_2\cdots i_\nu)\in \mathscr{X} $ is {\it  adapted to the valued quiver} $(Q, {\bf d})$ if  there exists an enumeration $\mathfrak{e}$ of $(Q, {\bf d})$ such that $\theta_\tau(\mathfrak{e}(\nu+1-k))=i_k, 1\leq k\leq \nu$.
 Note that if the word $(i_1i_2\cdots i_\nu)$ is adapt to $(Q, {\bf d})$, one can construct the enumeration $\mathfrak{e}$ from $(i_1i_2\cdots i_\nu)$ uniquely.
 We remark that our definition of adapted to a valued quiver is opposite to the one of ~\cite{Lusztig90}. In particular, a word ${\bf w}=(i_1\cdots i_\nu)$ is adapted to the valued quiver $(Q, {\bf d})$ in our definition if and only if ${\bf w}$ is adapted to the opposite valued quiver $(Q^{op}, {\bf d})$ in the sense of ~\cite{Lusztig90}.
 The following well-known fact gives the words of reduced expressions of $\omega_0\in W(C)$ which are adapted to the valued quiver $(Q, {\bf d})$~ ({\it cf.}~\cite{Lusztig90}).
 \begin{lemma}~\label{l:reducedexpression}
  For any enumeration $\mathfrak{e}$ of $(Q, {\bf d})$, define the word ${\bf w}_{\mathfrak{e}}:=(i_1,\cdots, i_\nu)$ by $\theta_\tau(\mathfrak{e}(\nu+1-k))=i_k, 1\leq k\leq \nu$. We have
  \[\omega_0=r_{\dimv \mathfrak{e}(\nu)}\circ\cdots \circ r_{\dimv \mathfrak{e}(2)}\circ r_{\dimv\mathfrak{e}(1)}=r_{i_1}\circ r_{i_2}\circ \cdots\circ r_{i_\nu}.
  \]
  In particular,  ${\bf w}_{\mathfrak{e}}$ is a word of a reduced expression of $\omega_0$ adapted to the valued quiver $(Q, {\bf d})$.
 \end{lemma}

 Let ${\bf w}=(i_1i_2\cdots i_\nu)$ be a word in $ \mathscr{X}$. If $(\alpha_{i_k}, \alpha_{i_{k+1}})=0$, then it is clear that ${\bf w}'=(i_1\cdots i_{k-1}i_{k+1}i_{k}i_{k+2}\cdots i_\nu)$ is also a word in $\mathscr{X}$. In this case, we say ${\bf w}'$ is obtained from ${\bf w}$ by a {\it short braid relation}. Let ${\bf w_1}=(i_1i_2\cdots i_\nu)$ and ${\bf w}_2=(j_1j_2\cdots j_\nu)$ be any two words of reduced expressions of $\omega_0$ which are adapted to the valued quiver $(Q,{\bf d})$. A well-known fact is that ${\bf w}_2$ can be obtained from ${\bf w}_1$ by  finitely many steps of short braid relations.
 Note that there are reduced expressions of $\omega_0$ which are not adapted to any valued quivers.

\subsection{Order $\leq_{\mathfrak{e}}$  of $\rep(Q, {\bf d})$ induced by an enumeration $\mathfrak{e}$}~\label{s:order}
From now on,  denote by $Q_0=I=\{1,2,\cdots,n\}$ and since $(Q, {\bf d})$ is acyclic, we may enumerate the vertices of $(Q, {\bf d})$ in  such a way that
\[\text{if there is an arrow from vertex $i$ to $j$, then $i<j$.}
\]

Fix an enumeration $\mathfrak{e}$ of $(Q, {\bf d})$ and denote by $\mathfrak{e}(k)=M_k$ for $1\leq k\leq \nu$. For each $M\in \rep(Q,{\bf d})$, there exists a unique vector $\underline{\mathfrak{e}}(M)=(a_1^M, \cdots, a_\nu^M)\in \N^\nu$ such that
\[M\cong M_1^{a_1^M}\oplus M_2^{a_2^M}\oplus \cdots \oplus M_\nu^{a_\nu^M}.
\]
For $M,N\in \rep(Q, {\bf d})$, we define
{\it $M<_{\mathfrak{e}}N$}~if there exists $1\leq k\leq \nu$ such that $a_1^M=a_1^N,\cdots, a_{k-1}^M=a_{k-1}^N$ and $ a_k^M<a_k^N$.
This defines a total order $\leq_{\mathfrak{e}}$ on $\rep(Q, {\bf d})$.
The following result shows that the order $\leq_{\mathfrak{e}}$ is compatible with quotients.
\begin{lemma}~\label{l:ordersurjective}
Let $X, Y, Z\in \rep(Q, {\bf d})$ such that there is a short exact sequence $0\to X\to Y\to Z\to 0$, then we have $Z\leq_{\mathfrak{e}} Y$.
\end{lemma}
\begin{proof}
We may write $X=M_1^{a_1^X}\oplus\cdots\oplus M_\nu^{a_\nu^X}, Y=M_1^{a_1^Y}\oplus\cdots\oplus M_\nu^{a_\nu^Y}$ and $Z=M_1^{a_1^Z}\oplus\cdots\oplus M_\nu^{a_\nu^Z}$. Set $Y_k:=M_1^{a_1^Y}\oplus \cdots\oplus M_k^{a_k^Y}$ and $Z_k:=M_1^{a_1^Z}\oplus\cdots\oplus M_k^{a_k^Z}$ for any $1\leq k\leq \nu$.

Applying $\Hom(-, M_1) $ to the  short exact sequence, we obtain the following exact sequence
\[0\to \Hom(M_1^{a_1^Z}, M_1)\to \Hom(M_1^{a_1^Y}, M_1)\to \Hom(M_1^{a_1^X}, M_1)\to \cdots,
\]
which implies that $a_1^Z\leq a_1^Y$. If $a_1^Z<a_1^Y$, we clearly have $Z\leq_{\mathfrak{e}} Y$ by the definition of $\leq_{\mathfrak{e}}$.  Now suppose that $a_1^Z=a_1^Y$ and apply the functor $\Hom(-, M_2)$ to the short exact sequence, we have
\[0\to \Hom(Z_2, M_2)\to \Hom(Y_2, M_2)\to \Hom(X, M_2)\to \cdots,
\]
from which we deduce that $a_2^Z\leq a_2^Y$. Repeating the above discussion yields the desired result.
\end{proof}

\subsection{Words arising from directed partitions}~\label{s:directedpartition}
A {\it directed partition $\mathcal{D}_*$ } of $\ind (Q, {\bf d})$ is a partition of the set $\ind (Q, {\bf d})$ into subsets $\mathcal{D}_1,\cdots, \mathcal{D}_s$ such that
\begin{itemize}
\item[(1)] $\Ext^1(U,V)=0$ for all $U,V$ in the same part $\mathcal{D}_k$;
\item[(2)] $\Hom(V,U)=0=\Ext^1(U,V)$ if $U\in \mathcal{D}_k, V\in\mathcal{D}_l$ and $1\leq k<l\leq s$.
\end{itemize}

The notion of directed partition was introduced in~\cite{Reineke} to construct monomial bases for $U_+$. Each enumeration $\mathfrak{e}$ of $(Q, {\bf d})$ yields a directed partition with $\nu$ subsets. On the other hand, for any given directed partition $\mathcal{D}_*=\mathcal{D}_1\cup\cdots\cup \mathcal{D}_s$, we may obtain an enumeration $\mathfrak{e}_{\mathcal{D}_*}$ (not unique) as follows:
   for each subset $\mathcal{D}_k$, we enumerate objects in $\mathcal{D}_k$ as $M_{k1},\cdots, M_{kt_k}, t_k\in \N$  such that $\Hom(M_{ki}, M_{kj})=0$ for $1\leq j<i\leq t_k$.
Then $M_{11}, M_{12},\cdots, M_{1t_1}, M_{21},\cdots, M_{s1},\cdots, M_{st_s}$ is an enumeration of $(Q, {\bf d})$ and we denote it by $\mathfrak{e}_{\mathcal{D}_*}$.

Let $\mathcal{D}_*$ be a given directed partition. For each $\mathcal{D}_k$, let $M_{\mathcal{D}_k}$ be the direct sum of one copy of each of the indecomposable objects in $\mathcal{D}_k$. Let $\mathcal{S}_k$ be the subset of $I=Q_0$ consisting of vertices $i$ such that $\Hom(P_i, M_{\mathcal{D}_k})\neq 0$. We define ${\bf w}_k=(i_{1}i_{2}\cdots i_{s_k})\in W$, where $i_{j}\in \mathcal{S}_k,~ 1\leq j\leq s_k,~ i_{1}<i_{2}<\cdots<i_{s_k}$ and $s_k=|\mathcal{S}_k|$. Finally, set ${\bf w}_{\mathcal{D}_*}=({\bf w}_1,\cdots, {\bf w}_k)\in W$ and  ${\bf w}_{\mathcal{D}_*}$ is called the {\it word associated to the directed partition} $\mathcal{D}_*$. Note that the length $l({\bf w}_{\mathcal{D}_*})$ of ${\bf w}_{\mathcal{D}_*}$ always satisfies $l({\bf w}_{\mathcal{D}_*})\geq \nu$.

For any $N\in \add M_{\mathcal{D}_k}$, we have
\[|N|=n_1|S_{i_1}|+n_2|S_{i_2}|+\cdots n_{s_k}|S_{i_{s_k}}|\in \go(Q, {\bf d}),~\text{where $n_j\geq 0$ for $1\leq j\leq s_k$}.
\]
We then define the vector $\underline{{\bf v}}_{\mathcal{D}_*}(N)$ as follows:
\[\underline{{\bf v}}_{\mathcal{D}_*}(N)=(0, \cdots, 0, n_1, n_2,\cdots, n_{s_k},0,\cdots, 0)\in \N^{l({\bf w}_{\mathcal{D}_*})},\]
 where $n_1$ is the $|\mathcal{S}_1|+\cdots+|\mathcal{S}_{k-1}|+1$ component of $\underline{{\bf v}}_{\mathcal{D}_*}(N)$. We call $\underline{{\bf v}}_{\mathcal{D}_*}(N)$ the {\it generated vector of $N$} with respect to the directed partition $\mathcal{D}_*$.

For a word ${\bf w}=(i_1i_2\cdots i_t)$ and $\underline{a}=(a_1a_2\cdots a_t)\in \N^t$, we define
\[\mathcal{E}(\underline{a}, {\bf w})=\{X\in \rep(Q, {\bf d})~|~F^X_{S_{i_1}^{a_1}, \cdots, S_{i_t}^{a_t}}\neq 0\}.
\]
\begin{lemma}~\label{l:order}
Let $N\in \add M_{\mathcal{D}_k}$ and $\underline{{\bf v}}_{\mathcal{D}_*}(N)\in \N^{l({\bf w}_{\mathcal{D}_*})}$ be the  generated vector of $N$ with respect to the directed partition $\mathcal{D}_*$.  Then we have
\begin{itemize}
\item[(1)] $N\in \mathcal{E}(\underline{{\bf v}}_{\mathcal{D}_*}(N), {\bf w}_{\mathcal{D}_*})$;
\item[(2)] For any $L\in \mathcal{E}(\underline{{\bf v}}_{\mathcal{D}_*}(N), {\bf w}_{\mathcal{D}_*})$, $N\leq_{\mathfrak{e}_{\mathcal{D}_*}} L$.
\end{itemize}
\end{lemma}
\begin{proof}
The statement $(1)$ follows from the definition of $\underline{{\bf v}}_{\mathcal{D}_*}(N)$ and $\mathcal{E}(\underline{{\bf v}}_{\mathcal{D}_*}(N), {\bf w}_{\mathcal{D}_*})$.

Let us prove the second statement. Suppose that there is an object $L\in \mathcal{E}(\underline{{\bf v}}_{\mathcal{D}_*}(N), {\bf w}_{\mathcal{D}_*})$ such that $L<_{\mathfrak{e}_{\mathcal{D}_*}}N$. Let $\mathcal{D}_k=\{M_{k1},\cdots, M_{kt}\}$ such that $\Hom(M_{ki}, M_{kj})=0$ for $1\leq j<i\leq t$. We may write $N$ as
\[N=M_{k1}^{a_1}\oplus\cdots\oplus M_{kt}^{a_t},~ a_i\in \N, ~1\leq i\leq t.
\]
Note that $\dimv L=\dimv N$ and  $L<_{\mathfrak{e}_{\mathcal{D}_*}}N$, we may also write $L$ as
\[L=M_{k1}^{b_1}\oplus\cdots\oplus M_{kt}^{b_t}\oplus L_0,
\]
where $L_0\in \add M_{\mathcal{D}_{k+1}}\oplus\cdots\oplus M_{\mathcal{D}_s}$ and there exists $1\leq l\leq t$ such that $a_1=b_1,\cdots, a_{l-1}=b_{l-1}, ~a_l>b_l$.

Now consider $N_1=M_{kl}^{a_l}\oplus\cdots\oplus M_{kt}^{a_t}$ and $L_1=M_{kl}^{b_l}\oplus\cdots\oplus M_{kt}^{b_t}\oplus L_0$, we have $\dimv N_1=\dimv L_1$. In particular, $\langle \dimv N_1, \dimv M_{kl}\rangle=\langle \dimv L_1,\dimv M_{kl}\rangle$.
However,
\begin{eqnarray*}
\langle \dimv N_1, \dimv M_{kl}\rangle&=&\dim_K\Hom(N_1, M_{kl})-\dim_K\Ext^1(N_1,M_{kl})\\
&=&a_l\dim_K\Hom(M_{kl}, M_{kl})\\
&>&b_l\dim_K\Hom(M_{kl}, M_{kl})-\dim_K\Ext^1(L_1, M_{kl})\\
&=&\dim_K\Hom(L_1, M_{kl})-\dim_K\Ext^1(L_1, M_{kl})\\
&=&\langle \dimv L_1, \dimv M_{kl}\rangle,
\end{eqnarray*}
a contradiction. This finishes the proof.
\end{proof}

Following ~\cite{Reineke}, given a representation $M\in \rep(Q, {\bf d})$,  we define its {\it $k$-th part} with respect to the directed partition $\mathcal{D}_*$ as
\[M_{(k)}:=\bigoplus_{U\in \mathcal{D}_k}U^{\mu_U M},
\]
where $\mu_UM$ denotes the multiplicity of the indecomposable $U$ as a direct summand of $M$. By the definition, we have
\begin{itemize}
\item[(1)] $M=\bigoplus\limits_{k=1}^sM_{(k)}$;
\item[(2)] $\Ext^1(M_{(k)},M_{(k)})=0$ for all $k=1,\cdots, s$;
\item[(3)] $\Ext^1(M_{(k)}, M_{(l)})=0=\Hom(M_{(l)}, M_{(k)})$ for all $1\leq k<l\leq s$.
\end{itemize}
Then we associate to $M$ a vector $\underline{{\bf v}}_{\mathcal{D}_*}(M)\in \N^{l({\bf w}_{\mathcal{D}_*})}$ by
\[\underline{{\bf v}}_{\mathcal{D}_*}(M)=\sum_{k=1}^s\underline{{\bf v}}_{\mathcal{D}_*}(M_{(k)}).
\]
We call $\underline{{\bf v}}_{\mathcal{D}_*}(M)$ the {\it generated vector of $M$} with respect to the directed partition $\mathcal{D}_*$. The following result plays  a key role in the study  of the injectivity of Feigin's map.
\begin{proposition}~\label{p:order}
For any $M\in \rep(Q, {\bf d})$, we have
\begin{itemize}
\item[(1)] $M\in \mathcal{E}(\underline{{\bf v}}_{\mathcal{D}_*}(M), {\bf w}_{\mathcal{D}_*})$;
\item[(2)] For any $L\in \mathcal{E}(\underline{{\bf v}}_{\mathcal{D}_*}(M), {\bf w}_{\mathcal{D}_*})$, $M\leq_{\mathfrak{e}_{\mathcal{D}_*}} L$.
\end{itemize}
\end{proposition}
\begin{proof}
First we note that
\[\mathcal{E}(\underline{{\bf v}}_{\mathcal{D}_*}(M), {\bf w}_{\mathcal{D}_*})=\{X\in \rep(Q, {\bf d})~|~F^X_{X_1,\cdots, X_s}\neq 0~\text{where}~X_k\in \mathcal{E}(\underline{{\bf v}}_{\mathcal{D}_*}(M_{(k)}), {\bf w}_{\mathcal{D}_*})\}.
\]
By Lemma~\ref{l:order}, we have $M_{(k)}\in \mathcal{E}(\underline{{\bf v}}_{\mathcal{D}_*}(M_{(k)}), {\bf w}_{\mathcal{D}_*})$. On the other hand, we clearly have $F^X_{M_{(1)}, \cdots, M_{(s)}}\neq 0$ if and only if $X=M_{(1)}\oplus\cdots\oplus M_{(k)}=M$. Moreover, in this case, we have $F^M_{M_{(1)}, \cdots, M_{(s)}}=1$. In particular, $M\in \mathcal{E}(\underline{{\bf v}}_{\mathcal{D}_*}(M), {\bf w}_{\mathcal{D}_*})$. This proves the statement $(1)$.

To prove statement $(2)$, let $L\in \mathcal{E}(\underline{{\bf v}}_{\mathcal{D}_*}(M), {\bf w}_{\mathcal{D}_*})$, there exists $L_i\in \mathcal{E}(\underline{{\bf v}}_{\mathcal{D}_*}(M_{(i)}), {\bf w}_{\mathcal{D}_*})$ for $1\leq i\leq s$ such that $F^L_{L_1,\cdots, L_s}\neq 0$. Note that if $L_1=M_{(1)},\cdots, L_s=M_{(s)}$, then $L\cong M$. Thus to prove the second statement, it suffices to prove that if there is an $1\leq i\leq s$ such that $L_i\not\cong M_{(i)}$, then $M<_{\mathfrak{e}_{\mathcal{D}_*}}L$.
In this case, we may assume that
\[L_1=M_{(1)}, \cdots, L_{k-1}=M_{(k-1)}, M_{(k)}<_{\mathfrak{e}_{\mathcal{D}_*}}L_k.
\]
Set $\underline{b}=\underline{{\bf v}}_{\mathcal{D}_*}(M_{k+1})+\cdots +\underline{{\bf v}}_{\mathcal{D}_*}(M_{s})$ and $X=M_{(1)}\oplus\cdots\oplus M_{(k-1)}$.
By the definition of $L$, we have  short exact sequences
\[0\to Y\to Z\to L_k\to 0~\text{and}~0\to Z\to L\to X\to 0,
\]
for some $Y\in \mathcal{E}(\underline{b}, {\bf w}_{\mathcal{D}_*})$.

Let $M_1,M_2,\cdots, M_\nu$ be the enumeration of $\ind(Q, {\bf d})$ associated to $\mathfrak{e}_{\mathcal{D}_*}$. We may write
\[X=M_1^{b_1}\oplus\cdots\oplus M_t^{b_t}~\text{and }~L=M_1^{a_1}\oplus\cdots\oplus M_\nu^{a_{\nu}},
\]
where $t=|\mathcal{D}_1|+\cdots+|\mathcal{D}_{k-1}|$.
By Lemma~\ref{l:ordersurjective}, we have $X\leq_{\mathfrak{e}_{\mathcal{D}_*}}L $ which implies that either $(i)$ there exists $1\leq l\leq t$ such that $a_1=b_1,\cdots, a_{l-1}=b_{l-1}$ and $a_l>b_l$ or $(ii)~ a_1=b_1,\cdots, a_t=b_t$. In case $(i)$, we clearly have $M<_{\mathfrak{e}_{\mathcal{D}_*}}L$. Let us consider the case $(ii)$. Denote by $L':=M_{t+1}^{a_{t+1}}\oplus\cdots\oplus M_\nu^{a_{\nu}}$, we have $\Hom(L', X)=0$, which implies that the short exact sequence $0\to Z\to L\to X\to 0$ is split and $L\cong Z\oplus X$.
Again by Lemma~\ref{l:ordersurjective}, we know that $M_{(k)}<_{\mathfrak{e}_{\mathcal{D}_*}}L_k\leq_{\mathfrak{e}_{\mathcal{D}_*}}Z$ and we deduce that $M<_{\mathfrak{e}_{\mathcal{D}_*}}L$.
\end{proof}

\begin{theorem}~\label{t:injective-directed-partition}
Let $\mathcal{D}_*$ be a directed partition and ${\bf w}_{\mathcal{D}_*}$ the associated word. The Feigin's map $\mathbf{F}_{{\bf w}_{\mathcal{D}_*}}:U_+\to \mathbb{P}_{{\bf w}_{\mathcal{D}_*}}$ is injective.
\end{theorem}
\begin{proof}
Let ${\bf w}_{\mathcal{D}_*}=(i_1i_2\cdots i_\mu)$ with $\mu=l({\bf w}_{\mathcal{D}_*})$.
By Theorem~\ref{t:Hallalgebra} $(1)$, it is equivalent to prove that $\int_{{\bf w}_{\mathcal{D}_*}}:\mathcal{H}^*(Q)\to \mathbb{P}_{{\bf w}_{\mathcal{D}_*}}$ is injective. Consider the adjoint of $\int_{{\bf w}_{\mathcal{D}_*}}$ ({\it cf.}~Section~\ref{ss:Feigin-type-maps}), which is
\begin{eqnarray*}
\Phi\circ \mathbf{T}_{{\bf w}_{\mathcal{D}_*}}:&\mathbb{P}_{{\bf w}_{\mathcal{D}_*}}^*&\to \mathcal{H}(Q)\\
&t_{\underline{a}}&\mapsto z_{{\bf w}_{\mathcal{D}_*}}(\underline{a}) [S_{i_1}]^{*a_1}*[S_{i_2}]^{*a_2}*\cdots*[S_{i_\mu}]^{*a_\mu},
\end{eqnarray*}
where $[S_{i_k}]^{*a_k}$ is the product of $a_k$ copies of $[S_{i_k}]$.

It suffices to show that $\Phi\circ \mathbf{T}_{{\bf w}_{\mathcal{D}_*}}$ is surjective. Note that $\{[M]~|~M\in \rep(Q, {\bf d})\}$ is a basis of $\mathcal{H}(Q)$, we are going to show that each $[M]$ belongs to the image of $\Phi\circ \mathbf{T}_{{\bf w}_{\mathcal{D}_*}}$.

Let $\underline{{\bf v}}_{\mathcal{D}_*}(M)=(a_1,\cdots, a_\mu)$ be the generated vector of $M$ with respect to $\mathcal{D}_*$, set
\[h_{{\bf w}_{\mathcal{D}_*}}(\underline{{\bf v}}_{\mathcal{D}_*}(M))=\prod_{k=1}^\mu~v_{i_k}^{\frac{a_k(a_k-1)}{2}}v^{\sum_{k<l}a_ka_l\langle \alpha_{i_k}, \alpha_{i_l}\rangle}.
\]
Proposition~\ref{p:order} and a direct calculation shows
\[\Phi\circ \mathbf{T}_{{\bf w}_{\mathcal{D}_*}}(t_{\underline{{\bf v}}_{\mathcal{D}_*}(M)})=h_{{\bf w}_{\mathcal{D}_*}}(\underline{{\bf v}}_{\mathcal{D}_*}(M))\left([M]+\sum_{M<_{\mathfrak{e}_{\mathcal{D}_*}}L}F^L_{S_{i_1}^{a_1},\cdots, S_{i_\mu}^{a_{\mu}}}[L]\right).
\]
Note that there are only finitely many $L$ such that $\dimv L=\dimv M$ and $M<_{\mathfrak{e}_{\mathcal{D}_*}}L$. Now induction on the order $\leq_{\mathfrak{e}_{\mathcal{D}_*}}$ implies that $[M]$ is in the image of $\Phi\circ\mathbf{T}_{{\bf w}_{\mathcal{D}_*}}$. This concludes the proof.

\end{proof}

We remark that in general the word ${\bf w}_{\mathcal{D}_*}$ does not belong to $\mathscr{X}$. In order to obtain ${\bf w}_{\mathcal{D}_*}\in \mathscr{X}$, one should consider the {\it regular directed partition} introduced in~\cite{Reineke}. Such a construction does generalize to the non-simply-laced cases. Therefore applying Theorem~\ref{t:injective-directed-partition} to a regular directed partition $\mathcal{D}_*$, we get the injectivity of $\mathbf{F}_{{\bf w}_{\mathcal{D}_*}}$ for the word ${\bf w}_{\mathcal{D}_*}$ of a reduced expression of $\omega_0$. We point out here that the proof of ${\bf w}_{\mathcal{D}_*}\in \mathscr{X}$ depends on Lusztig's work of canonical bases~\cite{Lusztig94}.
 On the other hand, for a directed partition $\mathcal{D}_*$, we may construct an enumeration $\mathfrak{e}_{\mathcal{D}_*}$ associated to $\mathcal{D}_*$ and the enumeration $\mathfrak{e}_{\mathcal{D}_*}$ yields a  word ${\bf w}_{\mathfrak{e}_{\mathcal{D}_*}}\in \mathscr{X}$ which is adapted to the valued quiver $(Q, {\bf d})$.
 Moreover, any words adpated to  vauled quivers can be constructed in this way.
 Thus, in the following, we will mainly restrict ourselves to the words adapted to valued quivers.

\subsection{Feigin's map for adapted words of  reduced expressions of  $\omega_{0}$}
Recall that $C\in M_{I\times I}(\Z)$ is a Cartan matrix of finite type and $W(C)$ is the associated Weyl group of $C$. Denote by $\omega_0$ the longest element of $W(C)$.
Let $U_+$ be the quantum enveloping algebra associated to $C$.

The following result is known to hold for any reduced expressions of $\omega_0$, which was conjectured by B. Feigin and proven by K. Iohara and F. Malikov in a special case~\cite{IM} and by A. Joseph in general cases~\cite{Joseph} ({\it cf.} also~\cite{Berenstein}). In ~\cite{Reineke}, M. Reineke also obtained a proof  by using the monomial bases constructed in ~\cite{Reineke} for simply-laced cases. Since we  pursue a representation-theoretic interpretation of the injectivity, we consider the reduced expressions which are adapted to valued quivers. The rest of this section is devoted to give a representation-theoretic proof for the following result.
\begin{theorem}~\label{t:injective}
Let ${\bf w}\in \mathscr{X}$ which is adapted to a valued quiver. Then the  Feigin's map $\mathbf{F}_{\bf w}:U_+\to \mathbb{P}_{\bf w}$ is injective.
\end{theorem}

Assume that ${\bf w}$ is adapted to the valued quiver $(Q, {\bf d})$.
By Theorem~\ref{t:Hallalgebra} $(1)$, we have $U_+\cong \mathcal{H}^*(Q)$.  Therefore to prove Theorem~\ref{t:injective},  it suffices to show that  $\int_{\bf w}:\mathcal{H}^*(Q)\to \mathbb{P}_{\bf w}$ is injective. Recall that we have the adjoint $\Phi\circ \mathbf{T}_{\bf w}:\mathbb{P}_{\bf w}^*\to \mathcal{H}(Q)$ of $\int_{\bf w}$~({\it cf.}~Section~\ref{ss:Feigin-type-maps}).
It is clear that Theorem~\ref{t:injective} follows from the   following result.
\begin{theorem}~\label{t:surjective}
Let ${\bf w}\in \mathscr{X}$ which is adapted to the valued quiver $(Q, {\bf d})$.
The map
\begin{eqnarray*}
\Phi\circ \mathbf{T}_{\bf w}:&\mathbb{P}_{\bf w}^*&\to \mathcal{H}(Q)\\
&t_{\underline{a}}&\mapsto \prod_{k=1}^\nu~v_{i_k}^{\frac{-a_k(a_k-1)}{2}}\frac{1}{[a_1]_{i_1}^!\cdots [a_\nu]_{i_\nu}^!}[S_{i_1}]^{*a_1}*\cdots *[S_{i_{\nu}}]^{*a_\nu}
\end{eqnarray*}
is surjective, where $[S_{i_k}]^{*a_k}$ is the product of $a_k$ copies of $[S_{i_k}]$.
\end{theorem}

\subsubsection{The behavior of $\Phi\circ \mathbf{T}_{\bf w}$ under short braid relations}
Let ${\bf w}_1$ and ${\bf w}_2$ be words of reduced expressions of $\omega_0$.
Denote by ${\bf w}_1=(j_1\cdots j_\nu)$ and ${\bf w}_2=(l_1\cdots l_\nu)$.
  Recall that we have the quantum polynomial algebras $\mathbb{P}_{{\bf w}_1}$ and $\mathbb{P}_{{\bf w}_2}$,
\[\mathbb{P}_{{\bf w}_1}=\langle t_1,\cdots, t_\nu~|~t_bt_a=v^{(\alpha_{j_b}, \alpha_{j_a})}t_at_b~\text{for}~b>a\rangle\]
\[\mathbb{P}_{{\bf w}_2}=\langle u_1,\cdots, u_\nu~|~u_bu_a=v^{(\alpha_{l_b}, \alpha_{l_a})}u_au_b~\text{for}~b>a\rangle.
\]
Assume moreover that ${\bf w}_2$ can be obtained from ${\bf w}_1$ by a short braid relation, say ${\bf w}_1=(j_1\cdots j_kj_{k+1}\cdots j_\nu)$ and ${\bf w}_2=(j_1\cdots j_{k+1}j_k\cdots j_\nu)$. In particular, $[S_{j_k}]$ and $[S_{j_{k+1}}]$ commute in $\mathcal{H}(Q)$. It is easy to see that the following assignment extends to an isomorphism of algebras
\begin{eqnarray*}
\Psi_{{\bf w}_1,{\bf w}_2}:&\mathbb{P}_{{\bf w}_1}&\to \mathbb{P}_{{\bf w}_2}\\
&t_i&\mapsto\begin{cases}u_{k+1}&i= k\\ u_{k}&i=k+1\\ u_i&\text{else}.
\end{cases}
\end{eqnarray*}
Moreover, the following diagram is commutative
\[\xymatrix{\mathbb{P}_{{\bf w}_2}^*\ar[d]_{\Psi_{{\bf w}_1,{\bf w}_2}^*}\ar[r]^{\Phi\circ \mathbf{T}_{{\bf w}_2}}&\mathcal{H}(Q)\\ \mathbb{P}_{{\bf w}_1}^*\ar[ur]_{\Phi\circ \mathbf{T}_{{\bf w}_1}},
}
\]
where $\Psi_{{\bf w}_1,{\bf w}_2}^*:\mathbb{P}_{{\bf w}_2}^*\to \mathbb{P}_{{\bf w}_1}^*$ is the adjoint of $\Psi_{{\bf w}_1,{\bf w}_2}$,  which is an isomorphism of coalgebras.

If  ${\bf w}_1$ and ${\bf w}_2$ are two words of  reduced expressions of $\omega_0$ which are adapted to a valued quiver $(Q, {\bf d})$,  we know that ${\bf w}_2$ can be obtained from ${\bf w}_1$ by finitely many steps of short braid relations. The above discussion implies that $\Phi\circ \mathbf{T}_{{\bf w}_1}$ is surjective if and only if $\Phi\circ \mathbf{T}_{{\bf w}_2}$ is surjective. In particular, to obtain Theorem~\ref{t:surjective}, it suffices to prove it for a special word ${\bf w}_0$ which is also adapted to the same valued quiver of ${\bf w}$.

\subsubsection{A special choice of the word ${\bf w}_0$}
Suppose that ${\bf w}$ is adapted to the valued quiver $(Q, {\bf d})$. Recall that $Q_0=I=\{1,\cdots, n\}$ and we have enumerated  the vertices of $Q$ in such a way that
\[\text{if there is an arrow from $i$ to $j$, then $i<j$.}
\]

Denote by $\mathcal{I}_1=\mathcal{I}=\{I_1,I_2,\cdots, I_n\}$ and $\mathcal{I}_k=\{\tau^{k-1}I_1,\tau^{k-1}I_2,\cdots, \tau^{k-1}I_n\}$ for each $k\in \N$. Since $|\ind(Q, {\bf d})|=\nu<\infty$, there is a unique $t\in \N$ such that $\mathcal{I}_{t}\neq \{0\}$ and $\mathcal{I}_{t+1}=\{0\}$. It is clear that $\tau^{t}I_n,\tau^{t}I_{n-1},\cdots, \tau^{t}I_1, \tau^{t-1}I_n,\cdots, I_n,\cdots, I_1$ yields an enumeration $\mathfrak{e}_{\mathcal{I}}$ of $\ind (Q,{\bf d})$ by deleting the zero representations in the sequence. Denote by ${\bf w}_0:={\bf w}_{\mathfrak{e}_{\mathcal{I}}}$ the corresponding word of $\mathfrak{e}_{\mathcal{I}}$ ({\it cf.} Lemma~\ref{l:reducedexpression}) which is also adapted to the valued quiver $(Q, {\bf d})$.
Thus to prove Theorem~\ref{t:injective}, it remains to prove Theorem~\ref{t:surjective} for this special word ${\bf w}_0$.

For any $1\leq k\leq~t$, we define a subset ${\bf S}_k$ of $I$ as
\[{\bf S}_k=\{j\in I~|~\tau^{k-1} I_j=0\}.
\]
By the definition, we have $\emptyset={\bf S}_1\subset {\bf S}_2\subset {\bf S}_2\subset\cdots\subset {\bf S}_{t}\subset I$.
For each $1\leq k\leq ~t$, we consider the following word
\[{\bf w}_{(k)}=(i_{k1}i_{k2}\cdots i_{kt_k}), ~\text{where $i_{k1}<i_{k2}<\cdots<i_{kt_k}$, $i_{kj}\in I\backslash {\bf S}_k$ and $t_k=| I\backslash {\bf S}_k|$.}
\]
In particular, ${\bf w}_{(1)}=(1 2\cdots  n)$ and ${\bf w}_0=({\bf w}_{(1)}, {\bf w}_{(2)},\cdots, {\bf w}_{(t)})$.

Denote by $\mathcal{P}_1=\mathcal{P}=\{P_1,P_2,\cdots, P_n\}$ and $\mathcal{P}_k=\{\tau^{-k+1}P_1, \tau^{-k+1}P_2,\cdots, \tau^{-k+1}P_n\}$ for $k\in \N$. Again by the finiteness of $|\ind (Q, {\bf d})|$, there is an $s\in \N$ such that $\mathcal{P}_{s}\neq \{0\}$ and $\mathcal{P}_{s+1}=\{0\}$.
It is clear that $\mathcal{P}_*=\mathcal{P}_1\cup\mathcal{P}_2\cup\cdots\cup \mathcal{P}_s$ is a directed partition of $\ind(Q, {\bf d})$.
\begin{lemma}~\label{l:number}
For each $k\in \N$, we have $|\mathcal{P}_k|=|\mathcal{I}_k|$ and $t=s$.
\end{lemma}
\begin{proof}
We prove the statement $|\mathcal{P}_k|=|\mathcal{I}_k|$ by induction on $k$. This is clear for $k=1$. Suppose that the equality holds for $k-1$. By the definition, we have
\[|\mathcal{P}_k|=|\mathcal{P}_{k-1}|-|\mathcal{P}_{k-1}\cap \mathcal{I}_1|~\text{and} ~|\mathcal{I}_k|=|\mathcal{I}_{k-1}|-|\mathcal{I}_{k-1}\cap \mathcal{P}_1|.
\]
Note that there is a bijection $g_{k-1}:\mathcal{P}_{k-1}\cap \mathcal{I}_1\to \mathcal{I}_{k-1}\cap \mathcal{P}_1$ given by $M\mapsto \tau^{k-2}M$ and the equality $|\mathcal{P}_k|=|\mathcal{I}_k|$ follows.

To show $t=s$, it is enough to notice that $|\ind(Q, {\bf d})|=\nu=|\mathcal{P}_1|+\cdots+ |\mathcal{P}_s|=|\mathcal{I}_1|+\cdots+|\mathcal{I}_t|$ and  $|\mathcal{P}_k|=|\mathcal{I}_k|$ for each $k\in \N$.

\end{proof}

\begin{lemma}~\label{l:dimensionvector}
For each $1\leq k\leq~t$, let $U_k=\opname{span}\{\dimv M~|~M\in \mathcal{P}_k\}\subseteq \Z^n$. Then the following equalities hold
\[U_k=\opname{span}\{\underline{e_i}~|~i\in I\backslash{\bf S}_k\}\subseteq \Z^n,
\]
where $\underline{e_1}, \underline{e_2},\cdots, \underline{e_n}$ is the standard basis of $\Z^n$.

\end{lemma}
\begin{proof}
For each $1<k\leq~t$, define
\[X_k:=\bigoplus_{j\in {\bf S}_k}I_j  ~\text{and}~ M_{\mathcal{P}_k}:=\bigoplus_{M\in \mathcal{P}_{k}}M.
\]
For any $i, j \in I$, we have $\Hom(\tau^{-k+1}P_i, I_j)\cong \Hom(P_i, \tau^{k-1}I_j)$. In particular, if $1<l\leq k$ and $j\in {\bf S}_l$, then $\Hom(\tau^{-k+1}P_i, I_j)=0$. Consequently, 
 \[\Hom(M_{\mathcal{P}_k}, X_j)=0 ~\text{for}~ 1< j\leq k,
 \]
which implies that $U_k\subseteq \opname{span}\{\underline{e_i}~|~i\in I\backslash{\bf S}_k\}$.

We claim that $T:=X_k\oplus M_{\mathcal{P}_k}$ is a tilting module for $\rep(Q, {\bf d})$. By Lemma~\ref{l:number}, we know that $|T|=n$.  It remains to show that $\Ext^1(T, T)=0$. Since  $X_k$ is injective, we have $\Ext^1(X_k, X_k)=0$  and $\Ext^1(M_{\mathcal{P}_k}, X_k)=0$. On the other hand, $\Ext^1(M_{\mathcal{P}_k}, M_{\mathcal{P}_k})=0$ since $\mathcal{P}_*$ is a directed partition.
 It remains to show $\Ext^1(X_k, M_{\mathcal{P}_k})=0$ and it suffices to show $\Ext^1(X, M)=0$ for any indecomposable direct summand $X$ of $X_k$ and  any indecomposable direct summand $M$ of $M_{\mathcal{P}_k}$. We may rewrite $X_k$ as
 \[X_k=(\bigoplus_{i\in {\bf S}_2\backslash {\bf S}_1}I_i)\oplus\cdots\oplus (\bigoplus_{i\in {\bf S}_k\backslash {\bf S}_{k-1}}I_i)
 \]
 Let $I_i$ be an indecomposable direct summand of $X_k$ such that $i\in {\bf S}_j\backslash{\bf S}_{j-1}$, where $2\leq j\leq k$. Let $M$ be an indecomposable direct summand of $M_{\mathcal{P}_k}$, there exists $i_M\in I$ such that $M=\tau^{-k+1}P_{i_M}$.  Note that $i\in {\bf S}_j$ implies that $\tau^{j-1}I_i=0$.
 By Auslander-Reiten duality, we have
 \[\Ext^1(I_i, M)\cong D\Hom(M, \tau I_i)=D\Hom(\tau^{-k+1}P_{i_M}, \tau I_i)=D\Hom(P_{i_M}, \tau^{k}I_i)=0,
 \]
 where $D=\Hom_K(-, K)$ is the $K$-duality.

Now $T$ is a tilting module implies that
\[\dimv I_j, ~j\in {\bf S}_k,~ \dimv M, ~M\in \mathcal{P}_k
\]
form a $\Z$-basis of $\Z^n$. In particular, for each $i\in I\backslash {\bf S}_k$, $\underline{e_i}$ is a $\Z$-linear combination of $\dimv I_j$ and $\dimv M$ for $j\in {\bf S}_k, ~M\in \mathcal{P}_k$, say
\[\underline{e_i}=\sum_{j\in {\bf S}_k}n_{ij}\dimv I_j+\sum_{M\in \mathcal{P}_k}n_{iM}\dimv M,
\]
where $n_{ij}, n_{iM}\in \Z$ for $j\in {\bf S}_k, M\in \mathcal{P}_k$. We claim that $n_{ij}=0$ for all $j\in {\bf S}_k$. Otherwise, let $l$ be the maximal element in ${\bf S}_k$ such that $n_{il}\neq 0$. Recall that $\Hom(M_{\mathcal{P}_k}, I_j)=0$ for all $j\in {\bf S}_k$ and $\Hom(I_j, I_l)=0$ for $j<l$. Considering the Euler bilinear form, we obtain
\begin{eqnarray*}
\langle \underline{e_i}, \dimv I_l\rangle&=&\sum_{j\in {\bf S}_k}n_{ij}\langle \dimv I_j, \dimv I_l\rangle+\sum_{M\in \mathcal{P}_k}n_{iM}\langle \dimv M, \dimv I_l\rangle\\
&=&n_{il}\dim_K\Hom(I_l, I_l)\neq 0.
\end{eqnarray*}
However, we have $\langle \underline{e_i}, \dimv I_l\rangle=\dim_K\Hom(S_i, I_l)=0$, a contradiction.
 In particular, $U_k=\opname{span}\{\underline{e_i}~|~i\in I\backslash{\bf S}_k\}$.

\end{proof}
 Recall that for the directed partition $\mathcal{P}_*$, we have a word ${\bf w}_{\mathcal{P}_*}$ associated to the directed partition $\mathcal{P}_*$ ({\it cf.}~Section~\ref{s:directedpartition}). As a consequence of Lemma~\ref{l:dimensionvector}, we have ${\bf w}_{\mathcal{P}_*}={\bf w}_0$.  By Theorem~\ref{t:injective-directed-partition}, $\mathbf{F}_{{\bf w}_0}$ is injective. This completes the proof of Theorem~\ref{t:injective}.

\section{Two consequences of the total  order}~\label{s:consequence}
\subsection{Monomial bases of $U_+$}
Let $C$ be a simply-laced Cartan matrix of finite type and $U_+$ the associated quantum enveloping algebra. In ~\cite{Reineke}, for a given directed partition $\mathcal{D}_{*}$, Reineke constructed a monomial basis for $U_+$. The key point in his construction is  the degeneration order $\leq_{\deg}$ which  is only valid for algebraically closed fields.  With the help of Frobenius morphisms, Deng and Du~\cite{DD} also constructed certain monomial bases for non-simply-laced quantum enveloping algebras of finite types via the degeneration order $\leq_{\deg}$.
In this subsection, we show that the order $\leq_{\mathfrak{e}_{\mathcal{D}_{*}}}$ we considered here can be used to replace the partial oder $\leq_{\deg}$ in~\cite{Reineke}. Thus it is possible to generalize most of the results in~\cite{Reineke} to all the Cartan matrices of finite type. Here we only consider the result of monomial bases.

Let $C\in M_{I\times I}(\Z)$ be a Cartan matrix of finite type and $U_+$ the associated quantum enveloping algebra. Let $(Q, {\bf d})$ be a valued quiver associated to  $C$. Recall that we have an isomorphism $U_+\cong \mathcal{H}(Q)$ by $E_i\mapsto [S_i]$.
In the following, we will identity $U_+$ with $\mathcal{H}(Q)$ by this isomorphism.

For $\underline{a}=(a_1,\cdots,a_n)\in \N^n$, set
\[E^{[\underline{a}]}:=E_1^{[a_1]}E_2^{[a_2]}\cdots E_n^{[a_n]}\in U_+.
\]
Let $\mathcal{D}_*=\mathcal{D}_1\cup\cdots\cup\mathcal{D}_s$ be a directed partition of $(Q, {\bf d})$.
For $M\in \rep(Q, {\bf d})$ with $\dimv M=(m_1,\cdots, m_n)$, we define a monomial as follows
\[E^{(M)}:=v^{-\dim_K\Hom(M, M)}\prod_{i=1}^nv_i^{m_i}E^{[\dimv M_{(1)}]}E^{[\dimv M_{(2)}]}\cdots E^{[\dimv M_{(s)}]}\in U_+.
\]

The following result has been proved in~\cite{Reineke} for simply-laced cases.
\begin{theorem}~\label{t:monomialbase}
$\{E^{(M)}~|~M\in \rep(Q, {\bf d})\}$ is a basis of $U_+$.
\end{theorem}
\begin{proof}
The proof follows the one of  ~\cite{Reineke}.
The idea  is to show that the matrix coefficients in the expansion of the $E^{(M)}$ in the basis $\{[M]~|~M\in \rep(Q, {\bf d})\}$ is lower unitriangular with respect to the order $\leq_{\mathfrak{e}_{\mathcal{D}_*}}$.
For completeness, we sketch the proof and we  separate the proof into  two claims.

{\bf Claim~1:} For each $\underline{a}=(a_1,\cdots, a_n)\in \N^n$, we have
\[E^{[\underline{a}]}=\prod_{i=1}^nv_i^{-a_i}\sum_{\dimv M=\underline{a}}v^{\langle |M|, |M|\rangle}[M].
\]
By a direct calculation, for each $1\leq i\leq n$, one has
\[E_i^{[a_i]}=\frac{[S_i]^{*a_i}}{[a_i]_i^{!}}=v_i^{a_i^2-a_i}[S_i^{a_i}].
\]
Thus,
\begin{eqnarray*}
E^{[\underline{a}]}&=&\prod_{i=1}^nv_i^{a_i^2-a_i} [S_1]^{a_1}*\cdots *[S_n]^{a_n}\\
&=&\prod_{i=1}^nv_i^{a_i^2-a_i} \sum_{[M]} v^{\sum_{k<l}a_ka_l\langle |S_k|, |S_l|\rangle}F^M_{S_1^{a_1},\cdots, S_n^{a_n}}[M].
\end{eqnarray*}
Note that if $F^M_{S_1^{a_1},\cdots, S_n^{a_n}}\neq 0$, then $F^M_{S_1^{a_1},\cdots, S_n^{a_n}}=1$ and in this case  we have $\dimv M=\underline{a}$, which implies that
\[v^{\langle|M|,|M|\rangle}=\prod_{i=1}^nv_i^{a_i^2}v^{\sum_{k<l}a_ka_l\langle|S_k|,|S_l|\rangle}.
\]
Putting all of these together, we obtain that $E^{[\underline{a}]}=\prod_{i=1}^nv_i^{-a_i}\sum_{\dimv M=\underline{a}}v^{\langle |M|, |M|\rangle}[M]$.

{\bf Claim~2:} For each $M\in \rep(Q, {\bf d})$, writing $E^{(M)}=\sum_{[L]}\gamma_L^M[L]$, we have $\gamma_L^M=0$ unless $M\leq_{\mathfrak{e}_{\mathcal{D}_*}}L $ and $\gamma_{M}^M=1$.

Denote by $\dimv M_{(k)}=(b_{k1},\cdots, b_{kn})\in \N^n$ and we have $\dimv M=\sum_{k=1}^s\dimv M_{(k)}=(m_1,\cdots, m_n)\in \N^n$.
We compute $E^{(M)}$ using Claim~1:
\begin{eqnarray*}
E^{(M)}&=&v^{-\dim_K\Hom(M,M)}\prod_{i=1}^nv_i^{m_i}\left(\prod_{i=1}^nv_i^{-b_{i1}}\sum_{\dimv N_1=\dimv M_{(1)}}v^{\langle|N_1|, |N_1|\rangle}[N_1]\right)\cdots \\ &&\left(\prod_{i=1}^nv_i^{-b_{s1}}\sum_{\dimv N_s=\dimv M_{(s)}}v^{\langle|N_s|, |N_s|\rangle}[N_s]\right)\\
&=&v^{-\dim_K\Hom(M,M)+\sum_{k=1}^s\langle|M_{(k)}|,|M_{(k)}|\rangle}\sum_{
\dimv N_i=\dimv M_{(i)},
1\leq i\leq s}[N_1]*[N_2]*\cdots* [N_s]\\
&=&v^{-\dim_K\Hom(M,M)+\sum_{k=1}^s\langle|M_{(k)}|,|M_{(k)}|\rangle+\sum_{k<l}\langle|M_{(k)}|,|M_{(l)}|\rangle}\\
&&\sum_{
\dimv N_i=\dimv M_{(i)},
1\leq i\leq s}\sum_{[L]}F^L_{N_1,\cdots, N_s}[L].
\end{eqnarray*}
Recall that we have $\Hom(M_{(j)}, M_{(i)})=0=\Ext^1(M_{(i)}, M_{(j)})$ for $1\leq i<j\leq s$ and $\Ext^1(M_{(i)}, M_{(i)})=0$ for all $1\leq i\leq s$.
We have
\[\dim_K\Hom(M,M)=\sum_{i=1}^s\dim_K\Hom(M_{(i)},M_{(i)})+\sum_{i<j}\dim_K\Hom(M_{(i)},M_{(j)}).\]
On the other hand,
$\langle |M_{(i)}|,|M_{(i)}|\rangle=\dim_K\Hom(M_{(i)},M_{(i)})$~and~\[\langle |M_{(i)}|, |M_{(j)}|\rangle=\dim_K\Hom(M_{(i)}, M_{(j)})~\text{for~$i<j$}.\]
In particular, we have proved that
\[E^{(M)}=\sum_{
\dimv N_i=\dimv M_{(i)},
1\leq i\leq s}\sum_{[L]}F^L_{N_1,\cdots, N_s}[L].
\]
Note that $F^L_{N_1,\cdots, N_s}\neq 0$ if and only if $L\in \mathcal{E}(\underline{{\bf v}}_{\mathcal{D}_*}(M), {\bf w}_{\mathcal{D}_*})$ and in this case we have $M\leq_{\mathfrak{e}_{\mathcal{D}_*}}L$ by Proposition~\ref{p:order}. Hence it remains to show the coefficient of $[M]$ in the right hand side is $1$. Note that $L=M$ if and only if $L_1\cong M_{(1)},\cdots, L_s\cong M_{(s)}$, and we have $F^M_{M_{(1)},\cdots, M_{(s)}}=1$. This finishes the proof of the Claim~$2$ and hence the proof of Theorem~\ref{t:monomialbase}.
\end{proof}

\subsection{A characterization of modules for representation-finite hereditary algebras}

Let $\Lambda$ be a finite-dimensional hereditary algebra over a  finite field $K$. Let $\mod \Lambda$ be the category of finite-dimensional right $\Lambda$-modules. Assume moreover that $\Lambda$ is representation-finite, {\it i.e.} there are only finitely many indecomposable $\Lambda$-modules up to isomorphism. Let ${\bf w}=(i_1i_2\cdots i_\nu)$ be the word of a reduced expression of the longest element  in the Weyl group associated to $\Lambda$. In ~\cite{Reineke}, Reineke observed  a direct consequence of the injectivity of Feigin's map, which says that each $\Lambda$-module $M$ is uniquely determined by the numbers $F^M_{S_{i_1}^{a_1}, \cdots, S_{i_\nu}^{a_\nu}}$ for $\underline{a}=(a_1,\cdots, a_\nu)\in \N^\nu$. The rest of this section is devoted to give a representation-theoretic interpretation for this fact.

Fix a directed partition $\mathcal{D}_*$ of $\Lambda$ and let ${\bf w}_{\mathcal{D}_*}=(i_1i_2\cdots i_\mu)$ be the word associated to $\mathcal{D}_*$.
For any $M\in \mod \Lambda$, we define a subset of $\N^\mu$ as follows
\[ \mathcal{S}(M, \mathcal{D}_*):=\{\underline{a}=(a_1,\cdots, a_\mu)\in \N^\mu~|~F^M_{S_{i_1}^{a_1}, \cdots, S_{i_\mu}^{a_\mu}}\neq 0\}.
\]
Note that the set $\mathcal{S}(M, \mathcal{D}_*)$ is always a finite set.
The following result inspired by Corollary $3.5$ of~\cite{Reineke}
shows that the module $M$ is uniquely determined by the set $\mathcal{S}(M,\mathcal{D}_*)$.
\begin{proposition}~\label{p:characterization}
Let $\Lambda$ be a representation-finite hereditary algebra over a finite field $K$. Let $\mathcal{D}_*$ be a directed partition of $\Lambda$ and ${\bf w}_{\mathcal{D}_*}$ the associated word with length $\mu$. Let $M, N$ be two $\Lambda$-modules. The following are equivalent:
\begin{itemize}
\item[(1)] $M\cong N$;
\item[(2)]  for any $\underline{a}=(a_1,\cdots, a_\mu)\in \N^{\mu}$, $F^M_{S_{i_1}^{a_1}, \cdots, S_{i_\mu}^{a_\mu}}=F^N_{S_{i_1}^{a_1}, \cdots, S_{i_\mu}^{a_\mu}}$;
\item[(3)] $\mathcal{S}(M,\mathcal{D}_*)=\mathcal{S}(N,\mathcal{D}_*)$.
\end{itemize}
\end{proposition}
\begin{proof}
It is clear that $(1)$ implies $(2)$ and $(2)$ implies $(3)$. It remains to show that $(3)$ implies $(1)$.
Let ${\bf \underline{v}}_{\mathcal{D}_*}(M)$ and ${\bf \underline{v}}_{\mathcal{D}_*}(N)$ be the generated vectors of $M$ and $N$ respectively. Thus we have
\[{\bf \underline{v}}_{\mathcal{D}_*}(M), {\bf \underline{v}}_{\mathcal{D}_*}(N)\in \mathcal{S}(M, \mathcal{D}_*)=\mathcal{S}(N, \mathcal{D}_*).
\]
Now ${\bf \underline{v}}_{\mathcal{D}_*}(M)\in \mathcal{S}(N, \mathcal{D}_*)$ implies that $N\in \mathcal{E}({\bf \underline{v}}_{\mathcal{D}_*}(M), {\bf w}_{\mathcal{D}_*})$ and we have $M\leq_{\mathfrak{e}_{\mathcal{D}_*}} N$ by Proposition~\ref{p:order}.
Similarly, by ${\bf \underline{v}}_{\mathcal{D}_*}(N)\in \mathcal{S}(M, \mathcal{D}_*)$ one deduces that $N\leq_{\mathfrak{e}_{\mathcal{D}_*}} M$. Thus, $M\cong N$.

\end{proof}

\def\cprime{$'$}
\providecommand{\bysame}{\leavevmode\hbox
to3em{\hrulefill}\thinspace}
\providecommand{\MR}{\relax\ifhmode\unskip\space\fi MR }
\providecommand{\MRhref}[2]{%
  \href{http://www.ams.org/mathscinet-getitem?mr=#1}{#2}
} \providecommand{\href}[2]{#2}


\begin{thebibliography}{10}

\bibitem{Berenstein}
A. Berenstein, \emph{Group-like elements in a quantum groups and Feigin's conjecture}, q-alg/9605015.

\bibitem{BR}
A. Berenstein and D. Rupel, \emph{Quantum cluster characters of Hall algebras},
Sel. Math. New Ser. \textbf{21}(2015), 1121-1176.

\bibitem{DR}
V. Dlab and C. M. Ringel, \emph{Indecomposable representations of graphs and algebras}, Memoirs Amer. Math. Soc. \textbf{173}(1976).

\bibitem{DD}
Bangming Deng and Jie Du, \emph{Tight monomials and the monomial basis property}, J. Algebra \textbf{324}(2010), 3355-3377.

\bibitem{Green}
J. A. Green, \emph{Hall algebras, hereditary algebras and quantun groups}, Invent. Math. \textbf{120}(1995), 2, 361-377.
\bibitem{Green97}
J. A. Green, \emph{Quantum groups, Hall algebras and quantum shuffles}, In: Finite reductive groups (Luminy 1994), 273-290, Birkh\"{a}user Prog. Math. \textbf{141}(1997).

\bibitem{IM}
K. Iohara and F. Malikov, \emph{Rings of skew polynomials and Gel'fand-Kirillov conjecture for quantum groups}, Comm. Math. Phys. \textbf{164}(1994), no. 2, 217-237.

\bibitem{Joseph}
A. Joseph, \emph{Sur une conjecture de Feigin}, C. R. Acd. Sci. Pairs \textbf{320}(1995), Serie I, 1441-1444.

\bibitem{Kassel}
C. Kassel, \emph{Quantum groups},  Graduate Texts in Math. \textbf{155}(1995),
Springer-Verlag New York.


\bibitem{Leclerc}
B. Leclerc, \emph{Dual canonical bases, quantum shuffles and $q$-characters}, Math. Zeit. \textbf{246}(2004), 691-732.


\bibitem{Lusztig90}
G. Lusztig, \emph{Canonical bases arising from quantized enveloping algebras}, J. Amer. Math. Soc. \textbf{3}(1990), no. 2, 447-498.
\bibitem{Lusztig93}
G. Lusztig, \emph{Introduction to quantum groups}, Progress in Mathematics \textbf{110}(1993).
\bibitem{Lusztig94}
G. Lusztig, \emph{Problems on canonical bases}, Proc. Symp. Pure Math. \textbf{56}(1994), 2, 169-176.

\bibitem{Reineke}
M. Reineke, \emph{Feigin's map and monomial bases for quantized enveloping algebras}, Math. Zeit. \textbf{237}(2001), 639-667.

\bibitem{Reineke10}
M. Reineke, \emph{Poisson automorphisms and quiver moduli}, J. Inst. Math. Jussieu \textbf{9} (2010), no. 3, 653Ð667.

\bibitem{Ringel}
C. M. Ringel, \emph{Hall algebras and quantum groups}, Invent. Math. \textbf{101}(1990), 583-592.
\bibitem{Rosso}
M. Rosso, \emph{Quantum groups and quantum shuffles}, Invent. math. \textbf{133}(1998), 399-416.

\bibitem{Rupel}
D. Rupel, \emph{The Feigin tetrahedron}, SIGMA. \textbf{11}(2015), 024, 30 pages.
\end{thebibliography}
\end{document}